\documentclass[11pt]{article}

\usepackage{amsmath,amsfonts,amsthm,amssymb,amscd}
\usepackage{url} 
\usepackage{bm} 
\usepackage{emptypage}

\usepackage[dvipsnames]{xcolor}

\usepackage{mathrsfs}
\usepackage{dsfont}
\usepackage{extarrows} 
\usepackage{bbold}
\usepackage[all]{xy}
\usepackage{stmaryrd}
\usepackage{graphicx}

\usepackage{geometry}
\usepackage[perpage]{footmisc}

\usepackage{graphicx}
\newcommand{\Ast}{\mathop{\scalebox{1.9}{\raisebox{1.07ex}{$\ast$}}}}%

\usepackage{tikz-cd}
\usetikzlibrary{matrix,arrows,decorations.pathmorphing}

\usepackage{mathtools}

\usepackage{hyperref}

\hypersetup{
colorlinks=true,
bookmarks=true,
bookmarksdepth=2,
bookmarksopen=false,
urlcolor=Mahogany,
anchorcolor=Mahogany,
citecolor=MidnightBlue,
linkcolor=Mahogany,
plainpages=false,
pagebackref=true,
pdfmenubar=true,
}

\theoremstyle{definition}
\newtheorem{thm}{Theorem}[section]
\newtheorem{lem}[thm]{Lemma}
\newtheorem{propo}[thm]{Proposition}
\newtheorem{cor}[thm]{Corollary}
\newtheorem{defn}[thm]{Definition}
\newtheorem{hyp}[thm]{Hypothesis}
\newtheorem{cons}[thm]{Construction}
\newtheorem{exmp}[thm]{Example}

\newtheorem{rmk}[thm]{Remark}

\DeclarePairedDelimiterX\SET[2]{\{}{\}}{\,#1 \;\delimsize\vert\; #2\,}
\DeclareMathOperator*{\stens}{\text{\raisebox{0.4ex}{\scalebox{0.67}{$\bigotimes$}}}}

\DeclareMathOperator*{\ssum}{\text{\raisebox{0.15ex}{\scalebox{0.78}{$\sum$}}}}
\DeclareMathOperator*{\bsum}{\text{\raisebox{0.15ex}{\scalebox{1.2}{$\sum$}}}}

\DeclareMathOperator*{\soplus}{\text{\raisebox{0.15ex}{\scalebox{0.85}{$\bigoplus$}}}}

\newcommand{\DD}{\mathds{D}}
\newcommand{\FF}{\mathds{F}}
\newcommand{\GG}{\mathds{G}}

\newcommand{\QQ}{\mathds{Q}}
\newcommand{\RR}{\mathds{R}}
\newcommand{\ZZ}{\mathds{Z}}

\newcommand{\NN}{\mathds{N}}

\newcommand{\dA}{\mathds{A}}

\newcommand{\kk}{\kappa}

\newcommand{\cE}{\mathcal{E}}
\newcommand{\cR}{\mathcal{R}}
\newcommand{\cO}{\mathcal{O}}
\newcommand{\cF}{\mathcal{F}}

\newcommand{\cW}{\mathcal{W}}
\newcommand{\cB}{\mathcal{B}}

\newcommand{\cHN}{\mathcal{HN}}

\newcommand{\fg}{\mathfrak{g}}
\newcommand{\fG}{\mathfrak{G}}

\newcommand{\fP}{\mathfrak{P}}

\newcommand{\bB}{\mathbf{B}}

\newcommand{\bv}{\mathbf{v}}
\newcommand{\bw}{\mathbf{w}}

\newcommand{\func}{\rule{2mm}{0.15mm}}

\DeclareMathOperator{\dR}{dR}

\renewcommand{\phi}{\varphi}

\DeclareMathOperator{\muu}{\boldsymbol\mu}

\DeclareMathOperator{\Aut}{Aut}

\DeclareMathOperator{\Diag}{Diag}

\DeclareMathOperator{\Frac}{Frac}
\DeclareMathOperator{\Gal}{Gal}
\DeclareMathOperator{\SL}{SL}

\DeclareMathOperator{\GL}{GL}
\DeclareMathOperator{\Lie}{Lie}

\DeclareMathOperator{\Spec}{Spec}

\DeclareMathOperator{\Hom}{Hom}
\DeclareMathOperator{\id}{\mathds {1}}

\DeclareMathOperator{\Mat}{Mat}
\DeclareMathOperator{\Der}{Der}
\DeclareMathOperator{\sep}{sep}

\DeclareMathOperator{\Ker}{Ker}
\DeclareMathOperator{\nr}{nr}

\DeclareMathOperator{\End}{End}

\DeclareMathOperator{\rk}{rk}

\DeclareMathOperator{\Supp}{Supp}

\DeclareMathOperator{\longto}{\longrightarrow}
\DeclareMathOperator{\la}{\langle}
\DeclareMathOperator{\ra}{\rangle}

\DeclareMathOperator{\fil}{fil}
\DeclareMathOperator{\gr}{gr} 

\newcommand{\lb}{(\!(} 
\newcommand{\rb}{)\!)}

\DeclareMathOperator{\uHom}{\underline{Hom}}
\DeclareMathOperator{\uIsom}{\underline{Isom}}

\DeclareMathOperator{\uAut}{\underline{Aut}}

\DeclareMathOperator{\Ob}{Ob}

\DeclareMathOperator{\Id}{Id} 
 
\DeclareMathOperator{\Rep}{\mathbf{Rep}}
\DeclareMathOperator{\Vect}{\mathbf{Vec}}
\DeclareMathOperator{\Mod}{\mathbf{Mod}}
\DeclareMathOperator{\Alg}{\mathbf{Alg}} 
 
\DeclareMathOperator{\forg}{forg}

\DeclareMathOperator{\HN}{HN} 
\DeclareMathOperator{\tR}{\tilde\cR} 
\DeclareMathOperator{\tE}{\tilde\cE}

\DeclareMathOperator{\I}{I}

\DeclareMathOperator{\GphimodR}{G-\mathbf{Mod}^{\phi}_R}
 
\DeclareMathOperator{\phimodR}{\mathbf{Mod}^{\phi}_R} 
\DeclareMathOperator{\phimodcR}{\mathbf{Mod}^{\phi}_\cR} 
\DeclareMathOperator{\phinabR}{\mathbf{Mod}^{\phi,\nabla}_R}

\DeclareMathOperator{\GphinabR}{G-\mathbf{Mod}^{\phi,\nabla}_R}

\DeclareMathOperator{\phinabcR}{\mathbf{Mod}^{\phi,\nabla}_{\cR}}

\DeclareMathOperator{\Fil}{\mathbf{Fil}} 
\DeclareMathOperator{\Grad}{\mathbf{Grad}}

\DeclareMathOperator{\Frob}{Frob}

\DeclareMathOperator{\dlog}{dlog} 
\DeclareMathOperator{\Ad}{Ad} 
\DeclareMathOperator{\fgl}{\mathfrak{gl}}

\author{Shuyang Ye}

\title{A group-theoretic generalization of the $p$-adic local monodromy theorem}

\begin{document}

\maketitle

\begin{abstract}
Let $G$ be a connected reductive group over a $p$-adic local field $F$. We propose and study the notions of $G$-$\phi$-modules and $G$-$(\phi,\nabla)$-modules over the Robba ring, which are exact faithful $F$-linear tensor functors from the category of $G$-representations on finite-dimensional $F$-vector spaces to the categories of $\phi$-modules and $(\phi,\nabla)$-modules over the Robba ring, respectively, commuting with the respective fiber functors. We study Kedlaya's slope filtration theorem in this context, and show that $G$-$(\phi,\nabla)$-modules over the Robba ring are ``$G$-quasi-unipotent'', which is a generalization of the $p$-adic local monodromy theorem proven independently by Y. Andr\'e, K. S. Kedlaya, and Z. Mebkhout.

\end{abstract}

\setcounter{tocdepth}{1}
{\hypersetup{hidelinks}
\tableofcontents
}

\section{Introduction}

Let $p$ be a prime number and $q$ a power of $p$. Let $K$ be a complete non-archimedean discretely valued field of characteristic $0$ equipped with an isometric automorphism $\phi$, the \emph{Frobenius}, inducing the $q$-power map on the residue field $\kk\supseteq \FF_q$. We also require $K$ to be unramified over the fixed subfield $F$ under $\phi$. See Hypothesis~\ref{H:FK} for a concrete example.

The \emph{Robba ring} $\cR=\cR(K,t)$ is the ring of bidirectional power series $\sum\limits_{i\in\ZZ}c_it^i$ in one variable $t$ with coefficients in $K$ which converge in an annulus $[\alpha,1)$ for some series-dependent $0<\alpha<1$. The Robba ring $\cR$ is endowed with an absolute Frobenius lift $\phi$ which extends the Frobenius on $K$ and lifts the $q$-power map on $\kk \lb t \rb$, and with the derivation $\partial=d/dt$.

A \emph{$(\phi,\nabla)$-module} over $\cR$ is a triple $(M,\Phi,\nabla)$, where $M$ is a finite free $\cR$-module, $\Phi$ is a \emph{Frobenius}, i.e.\ a $\phi$-linear endomorphism of $M$  whose image spans $M$ over $\cR$, and $\nabla\colon M\to M\bigotimes_\cR \cR dt$ is a connection. Moreover, $\Phi$ and $\nabla$ should satisfy the \emph{gauge compatibility condition}, which says that, after choosing an $\cR$-basis for $M$ the actions $\Phi$ and $\nabla$ are given by matrices $A$ and $N$ respectively, and these matrices should satisfy $N=\muu\cdot A(\phi(N))A^{-1}-\partial(A)A^{-1}$, where $\muu\coloneqq \partial(\phi(t))$.

The $(\phi,\nabla)$-modules, also known as the \emph{overconvergent ($\mathrm{F}$-)isocrystals} in the literature, are closely related to $p$-adic local systems on $\Spec \kk \lb t \rb$ (for a summary, we refer to~\cite{ked-iso}), for which the correct monodromy theorem is the \emph{$p$-adic local monodromy theorem} ($p$LMT), proven independently by Andr\'e~\cite{andre}, Kedlaya~\cite{ked-ann} and Mebkhout~\cite{meb}. It states that every $(\phi,\nabla)$-module over $\cR$ is quasi-unipotent. Concretely, a $(\phi,\nabla)$-module $M$ over $\cR$, after an \'etale extension to $\cR_L$ (the Robba ring canonically associated to some finite separable extension $L$ of $\kk \lb t \rb$), admits a filtration by sub-$(\phi,\nabla)$-modules such that the connections induced on the gradiation are trivial. A matricial description of the theorem is given as follows. Let $d$ be the rank of $M$ over $\cR$, and let $A\in \GL_d(\cR)$ (resp.\ $N\in \Mat_{d,d}(\cR)$) be the matrix of $\Phi$ (resp.\ $\nabla$) in some basis. Then there exists $U\in \GL_d(\cR_L)$ such that $U^{-1}NU-U^{-1}\partial(U)$ is an upper-triangular block matrix with zero blocks in the diagonal.

We mention two applications of the $p$LMT in $p$-adic Hodge theory. 
\begin{itemize}
\item[$\bullet$] In~\cite{ber-diff}, Berger associated to every $p$-adic de Rham representation $V$ a $(\phi,\nabla)$-module $\mathrm{N}_{\dR}(V)$ over $\cR$. Using the $p$LMT, he proved the $p$-adic monodromy theorem (previously a conjecture of Fontaine): every $p$-adic de Rham representation is potentially semistable. 
\item[$\bullet$] In~\cite{mar}, Marmora used the $p$LMT to construct a functor from the category of $(\phi,\nabla)$-modules over $\cR$ to that of $K^{\nr}$-valued Weil-Deligne representations of the Weil group $\cW_{\kk \lb t \rb}$, where $K^{\nr}$ is the maximal unramified extension of $K$ in a fixed algebraic closure of $K$.
\end{itemize}

Rather than the general linear group, a Galois representation may take values in some connected reductive group, such as the special linear group or the symplectic group. In order to have appropriate formulations of the above results in this context, it is helpful to establish a $G$-version of the $p$LMT for a connected reductive group $G$, which is the main motivation of our present paper. 

In this paper, we introduce the notion of \emph{$G$-$\phi$-modules over $\cR$} (resp.\ \emph{$G$-$(\phi,\nabla)$-modules over $\cR$}), which are exact faithful $F$-linear tensor functors from the category $\Rep_F(G)$ of $G$-representations on finite-dimensional $F$-vector spaces to the category $\phimodcR$ of $\phi$-modules over $\cR$ (resp.\ to the category $\phinabcR$ of $(\phi,\nabla)$-modules over $\cR$), commuting with the respective fiber functors. These constructions are inspired by that of \emph{$G$-isocrystals} introduced in~\cite[\S~IX. 1]{dat}. Our main result is the following $G$-version of the $p$LMT.

\begin{thm}[Theorem~\ref{T:G-mono}]\label{T':G-mono}
Let $G$ be a connected reductive $F$-group and let $\fg$ be its Lie algebra. If $g\in G(\cR)$ and $X\in \fg\otimes_F\cR$ satisfy the gauge compatibility condition $X=\Gamma_g\big(\mu\phi(X) \big)$, then there exists a finite separable extension $L$ over $\kk \lb t \rb$ and an element $b\in G(\cR_L)$ such that $\Gamma_b(X)\in \Lie \big(U_{G_\cR}(-\lambda_g)\big)\bigotimes_\cR {\cR_L}$.
\end{thm}

Here, $\Gamma_y(Y)=\Ad(y)(Y)-\dlog(y)$ for all $y\in G(\cR)$ and $Y\in \fg\otimes_F \cR$, and $\lambda_g \colon \GG_{m,\cR} \to G_\cR$  is a cocharacter associated to $g$ whose inverse is denoted by $-\lambda_g$. For example, $\Ad(y)(Y)=yYy^{-1}$ and $\dlog(y)=\partial(y)y^{-1}$, when $G=\GL_d$. In this context, $U_{G_\cR}(-\lambda_g)$ denotes the unipotent radical of the parabolic subgroup of $G_\cR$ associated to $-\lambda_g$.

When $G=\GL_d$, $g$ (resp.\ X) should be thought as the matrix of the Frobenius (resp.\ the matrix of the connection), and $\Gamma_b(\func)$ as the matrix of a connection under the change-of-basis via $b^{-1}$, in particular, the gauge compatibility condition coincides with the matricial one given before. Moreover, $\Lie \big(U_{G_\cR}(-\lambda_g)\big)\stens_\cR {\cR_L}$ consists of upper-triangular matrices over $\cR_L$ with zero-blocks (of certain sizes) in the diagonal. As such, Theorem~\ref{T':G-mono} recovers the matricial $p$LMT described above.

In Proposition~\ref{P:B}, we show that $G$-$(\phi,\nabla)$-modules over $\cR$ are indeed pairs $(g,X)$ subject to the gauge compatibility condition in the theorem. In this sense, the theorem can be interpreted as saying that $G$-$(\phi,\nabla)$-modules over $\cR$ are ``$G$-quasi-unipotent".

Our approach to the theorem closely follows that of the $p$LMT in~\cite{ked-ann} for absolute Frobenius lifts, wherein the author used his slope filtration theorem (along with applying the pushforward functor and twisting to each quotient of the filtration) to reduce the problem to the unit-root case, and then apply the unit-root $p$LMT attributed to Tsuzuki~\cite{tsu-mono} to finish. More precisely, we use Kedlaya's slope filtration theorem to construct a $\QQ$-filtered fiber functor $\HN_g$ from $\Rep_F(G)$ to  $ \QQ$-$\Fil_\cR$, the category of $\QQ$-filtered modules over $\cR$ (see Theorem~\ref{T:filtered-fiber}). We then reduce $\HN_g$ to a $\ZZ$-filtered fiber functor $\HN_g$ from $\Rep_F(G)$ to $\ZZ$-$\Fil_\cR$, the category of $\ZZ$-filtered modules over $\cR$ (see Lemma~\ref{L:Q-to-Z}). Then a result of Ziegler (Theorem~\ref{T:ziegler}) immediately implies that $\HN_g^\ZZ$ is \emph{splittable}, i.e.\ factors through a $\ZZ$-graded fiber functor (see Proposition~\ref{P:Z-splitting}). In particular, for any splitting of $\HN_g^\ZZ$, we construct a morphism $\lambda_g\colon \GG_{m,\cR} \to G_\cR$ of $\cR$-groups in \S~\ref{S:slope-morphism}, which is called the \emph{$\ZZ$-slope morphism} of $g$. With this, we can reduce the $G$-$(\phi,\nabla)$-module $(g,X)$ over $\cR$, involving the (generalized) pushforward functor and twisting, to a unit-root one (see Corollary~\ref{C:unit-root}). Theorem~\ref{T':G-mono} then follows from the unit-root $p$LMT and a tannakian argument.

The paper is organized as follows. In Section~\ref{S:pre}, we set up basic notation and conventions, and then recall some necessary background on the theory of slopes and tannakian formalism. In Section~\ref{S:G-phi}, we study $G$-$\phi$-modules over the Robba ring, and construct slope morphisms. In Section~\ref{C:G-phi-nabla}, we consider $G$-$(\phi,\nabla)$-modules over the Robba ring, and prove our main result, Theorem~\ref{T':G-mono}, in the last subsection.

\subsection*{Acknowledgement}
The content of this paper is part of the author's Ph.D. thesis carried out at Humboldt-Universit\"at zu Berlin. The author owes a deep gratitude to his supervisor Elmar Gro\ss e-Kl\"onne for providing him this problem, and for all the helpful discussions. The author would like to thank the external examiners of the thesis for their valuable feedback.  The author is also indebted to Claudius Heyer for many constructive suggestions.

\section{Preliminaries}\label{S:pre}

\subsection{Notation and conventions}

When $k$ is a field, we denote by $\Vect_k$ the category of finite-dimensional $k$-vector spaces. When $R\in$ is a $k$-algebra\footnote{By an algebra, we always mean a commutative algebra with $1$.}, we denote by $\Mod_R$ the category of $R$-modules, and by $\Alg_R$ the category of $R$-algebras. When $V,W \in \Vect_k$, we write $V_R$ for $V\stens_k R$, and write $\alpha_R\coloneqq \alpha\otimes R$, the $R$-linear extension of $\alpha$, for all $k$-linear maps $\alpha\colon V\to W$. When $G$ is an affine algebraic $k$-group, we denote by $k[G]$ the Hopf algebra of $G$, by $G_R\coloneqq G \times_{\Spec k} \Spec R$ the base extension, by $H^1(G,k)\coloneqq H^1 \big(\Gal(k^{\sep}/k),G(k^{\sep}) \big)$ the first Galois cohomology set, and by $\Rep_k(G)$ the category of representations of $G$ on finite-dimensional $k$-vector spaces. We denote by $\omega^G$ the (forgetful) fiber functor $\Rep_k(G)\to \Vect_k$.

By a reductive $k$-group, we mean a (not necessarily connected) affine algebraic $k$-group $G$ such that every smooth connected unipotent normal subgroup of $G_{\bar k}$ is trivial, where $\bar k$ is an algebraic closure of $k$.

For the rest of this paper, we work under the following hypothesis. 

\begin{hyp}\label{H:FK}
Let $p$ be a prime number and $q=p^f$ an integral power of $p$. Let $F$ be a finite extension of $\QQ_p$ with the ring of integers $\cO_F$, a fixed uniformizer $\pi_F$ and the residue field $\kk_F$ of $q$ elements. Let $\kk$ be a perfect field containing $\kk_F$. Let $\cO_K=\cO_F\stens_{W(\kk_F)} W(\kk)$, where $W(\kk_F)$ (resp.\ $W(\kk)$) denotes the ring of Witt vectors with coefficients in $\kk_F$ (resp.\ in $\kk$). Then $K\coloneqq \Frac(\cO_K)\cong F\stens_{W(\kk_F)} W(\kk)$ is a complete discretely valued field with ring of integers $\cO_K$, a uniformizer $\pi\coloneqq \pi_F\otimes 1$ and residue field $\kk$. Let $\Frob$ be the ring endomorphism of $W(\kk)$ induced by the $p$-power map on $\kk$, and let
\[\phi\coloneqq \Id_F\otimes \Frob^f \colon K\longto K\]
be the \emph{Frobenius automorphism} on $K$ relative to $F$. Then $\phi$ reduces to the $q$-power map on $\kk$, and the fixed field of $\phi$ on $K$ is $F\stens_{W(\kk_F)} W(\kk_F) \cong F$.
\end{hyp}

\subsection{The Robba ring and its variants}\label{S:robba}

For $\alpha\in(0,1)$, we put
\[\cR_\alpha\coloneqq \Big\{\ssum\limits_{i\in\ZZ} c_it^i~\Big|~\ c_i\in K, \lim\limits_{i\to\pm\infty} |c_i|\rho^i=0,~\forall\rho\in[\alpha,1) \Big\}.\]
For any $\rho\in[\alpha,1)$, we define the $\rho$-Gauss norm on $\tilde{\cR}_\alpha$ by setting $\big|\sum\limits_i c_it^i \big|_\rho \coloneqq\sup_i\{|c_i|\rho^i\}$.
The \emph{Robba ring} is defined to be the union $\cR \coloneqq \cR(K,t)\coloneqq \bigcup\limits_{\alpha\in(0,1)}\cR_\alpha$. For any $\sum\limits_i c_it^i\in\cR$, we define $\big|\sum\limits_i c_it^i\big|_1:=\sup_i\{|c_i|\}\in \RR_{\geq 0}\cup \{\infty\}$, the $1$-Gauss norm.

The \emph{bounded Robba ring} $\cE^\dagger=\cE^\dagger(K,t)$ is the subring of $\cR$ consisting of bounded elements (i.e.\ elements with finite $1$-Gauss norm), which is actually a henselian discretely valued field w.r.t.\ the $1$-Gauss norm with residue field $\kk \lb t\rb$.

Let $R\in\{\cR, \cE^\dagger\}$. An \emph{absolute $q$-power Frobenius lift} on $R$ is a ring endomorphism $\phi \colon R\to R$ given by $\sum\limits_{i\in \ZZ} c_it^i\longmapsto \sum\limits_{i\in \ZZ} \phi(c_i)u^i$.

For any $\alpha\in(0,1)$, we define $\tilde{\cR}_\alpha$ to be the ring of formal sums $\ssum\limits_{i\in\QQ} c_it^i$ with $c_i\in K$, subject to the following properties.
\begin{itemize}
\item  For any $c>0$, the set $\{i\in\QQ\mid |c_i|\geq c\}$ is well-ordered.
\item For any $\rho\in[\alpha,1)$, we have $\lim\limits_{i\to\pm\infty} |c_i|\rho^i=0$.
\end{itemize}
For any $\rho\in[\alpha,1)$, we define the $\rho$-Gauss norm on $\tilde{\cR}_\alpha$ by setting
\[\Big|\ssum\limits_i c_it^i \Big|_\rho=\sup\limits_{i\in\QQ}\{|c_i|\rho^i\}.\]
We define $\tilde{\cR} \coloneqq \tilde{\cR}(K,t)=\bigcup\limits_{\alpha\in(0,1)} \tilde{\cR}_\alpha$, the \emph{extended Robba ring}. The \emph{absolute Frobenius lift} on $\tilde{\cR}$ is a ring automorphism on $\tR$ given by $\sum\limits_{i\in\QQ} c_it^i \mapsto \sum\limits_{i\in\QQ} \phi
(c_i)t^{iq}$. We denote by $\tE^\dagger$ the subring of $\tR$ consisting of bounded elements. By~\cite[Proposition 2.2.6]{ked-relative}, we have a $\phi$-equivariant embedding $\psi:\cR\to\tilde{\cR}$ such that $|\psi(x)|_\rho=|x|_\rho$ for $\rho$ sufficiently close to $1$.

\subsection{The slope filtration theorem}\label{S:slope}

We recall Kedlaya's theory of slopes. Let $R\in\{\cE^\dagger,\cR,\tE^\dagger,\tR\}$ equipped with a Frobenius lift $\phi$. For the notions of $\phi$-modules and $(\phi,\nabla)$-modules over $R$, we refer to~\cite[\S 2.5]{ked-ann}. We denote by $\phimodR$ (resp.\ $\phinabR$) the category of $\phi$-modules (resp.\ $(\phi,\nabla)$-modules) over $R$.

Let $(M,\Phi)\in \phimodR$ and let $n$ be a positive integer. Then $(M,\Phi^n)$ is a $\phi^n$-module over $(R,\phi^n)$. The $n$-\emph{pushforward functor} is given by
\[[n]_*\colon \phimodR \longto \mathbf{Mod}^{\phi^n}_R, \;\;\;\;\; (M,\Phi) \longmapsto (M,\Phi^n).\]
For any $s\in \ZZ$, we define the \emph{twist} $M(s)$ of $(M,\Phi)$ by $s$ to be the $\phi$-module $(M,\pi^s\Phi)$. Now let $M$ be a $\phi$-module over $R$ of rank $d$.
\begin{itemize}
\item[(i)] We say that $M$ is \emph{unit-root} $\phi$-module if there exists a basis $\bv_1,\cdots,\bv_d$ of $M$ over $R$ in which $\Phi$ acts via an invertible matrix in $\GL_d(\cO_{\cE^\dagger})$ if $R\in \{\cE^\dagger,\cR\}$, or $\GL_d(\cO_{\tE^\dagger})$ if $R\in\{\tE^\dagger,\tR\}$.
\item[(ii)] Let $\mu=s/r \in\QQ$ with $r>0$ and $(s,r)=1$. We say that $M$ is \emph{pure of slope} $\mu$ if $([r]_*M)(-s)$ is unit-root.
\end{itemize}

Let $M$ be a $\phi$-module over $\cR$. By Kedlaya's \emph{slope filtration theorem} on $\phi$-modules (\cite[Theorem 6.10]{ked-ann}). We have a canonical filtration $0=M_0 \subseteq M_1 \subseteq \cdots \subseteq M_l=M$ of sub-$\phi$-modules over $\cR$ such that each quotient $M_i/M_{i-1}$ is pure of some slope $\mu_i$ with $\mu_1< \cdots <\mu_l$, which is called the \emph{slope filtration} of $M$. We call $\mu_1,\cdots,\mu_l$ the \emph{jumps} of the slope filtration. The (uniquely determined, not neccesarily strictly) increasing sequence $(\mu_1,\cdots,\mu_1,\cdots,\mu_l,\cdots,\mu_l)$, with each $\mu_i$ appearing $\rk_\cR(M_i/M_{i-1})$ times, is said to be the \emph{Newton slope sequence} for $M$. We call $\rk_\cR(M_i/M_{i-1})$ the \emph{multiplicity} of $\mu_i$ for all $1\leq i\leq l$. Moreover, if $M$ is a $(\phi,\nabla)$-module over $\cR$, then the slope filtration can be refined to a filtration of sub-$(\phi,\nabla)$-modules. This is~\cite[Theorem 6.12]{ked-ann}, and is referred to the \emph{slope filtration theorem for $(\phi,\nabla)$-modules}.

We next recollect some results on the theory of slopes for later use.

\begin{lem}\label{L:nonzero-morphism}
Let $R\in \{\cR,\tR\}$ and let $M$ and $N$ be $\phi$-modules over $R$. If the slopes of $M$ are all less than the smallest slope of $N$, then no non-zero morphism from $M$ to $N$ exists.
\end{lem}
\begin{proof}
This is~\cite[Proposition 1.4.18]{ked-relative}.
\end{proof}

\begin{lem}[{\cite[Lemma 1.5.3]{liu-families}}]\label{L:admissible}
The field $K$ admits an admissible extension $E$ such that the residue field $\kk_E$ of $E$ is strongly difference-closed.
\end{lem}

We need only the following consequences of the existence of such an $E$; the notion of \emph{admissible extensions} or \emph{strong difference-closeness} will not be explicitly used in this paper, for which we refer to loc.\ cit.. See also~\cite{ked-relative} (see in particular Hypothesis 2.1.1 for the condition of being strongly difference-closed).

\begin{lem}\label{L:tilde-tensor}
Let $E$ be an admissible extension of $K$ such that $\kk_E$ is strongly difference-closed.
\begin{itemize}
\item[(i)] Let $M \in \phimodcR$. If $M$ is pure of some slope $\mu$, then $M \stens_\cR \tilde\cR(E,t)$ is pure of slope $\mu$. 
\item[(ii)] Let $M \in \phimodcR$. Then tensoring the slope filtration of $M$ with $\tilde\cR(E,t)$ gives the slope filtration of $M\stens_{\cR}\tilde\cR(E,t)$.
\item[(iii)] Let $0 \longto M_1 \longto  M \longto M_2 \longto 0$
be a short exact sequence of $\phi$-modules over $\tilde\cR(E,t)$ such that the slopes of $M_1$ are all less than the smallest slope of $M_2$. Then the sequence splits.
\item[(iv)] Every $\phi$-module over $\tilde\cR(E,t)$ admits a \emph{Dieudonn\'e-Manin decomposition},~i.e.\ is a direct sum of standard $\phi$-submodules.
\end{itemize}
\end{lem}
\begin{proof}
Assertion (i) is immediate from ~\cite[Theorem 3.1.3]{ked-relative}. For assertion (ii), we let $M$ be a $\phi$-module over $\cR$. Then $M \otimes_{\cR} \tilde\cR(E,t)$ is also semistable by~\cite[Theorem 3.1.2]{ked-relative}. Since $\kk_E$ is strongly difference-closed by assumption, we have that $M \otimes_{\cR} \tilde\cR(E,t)$ is pure of some slope by~\cite[Theorem 2.1.8]{ked-relative}. It follows from assertion (i) that $M$ is pure of the same slope, assertion (ii) then follows. Assertion (iii) is~\cite[Proposition 1.5.11]{liu-families}, and Assertion (iv) is Proposition 1.5.12 in loc.\ cit..

\end{proof}

\subsection{The tannakian duality}
In this subsection, $k$ denotes a field.
We follow the definitions and notations in~\cite{tc}.The following \emph{tannakian duality} will be repeatedly used in this paper, whose proof can be found, e.g.\ in~\cite[Theorem 9.2]{lag}.

\begin{thm}\label{T:recover}
Let $G$ be an affine algebraic $k$-group and let $R\in\Alg_k$. Suppose that for any $(V,\rho_V)\in\Rep_k(G)$ we are given an $R$-linear map $\lambda_V\colon V_R\to V_R$. If the family $\{\lambda_V\mid (V,\rho_V)\in \Rep_k(G)\}$ satisfies
\begin{itemize}
\item[(i)] $\lambda_{V\stens W}=\lambda_V\otimes\lambda_W$ for all $V,W\in\Rep_k(G)$;
\item[(ii)] $\lambda_{\id}$ is the identity map where $\id$ is the trivial representation on $k$;
\item[(iii)] for all $G$-equivariant maps $\alpha\colon V\to W$, we have $\lambda_W\circ \alpha_R=\alpha_R\circ\lambda_V$.
\end{itemize}
Then there exists a unique $g\in G(R)$ such that $\lambda_V=\rho_V(g)$ for all $V$.
\end{thm}

\begin{cor}\label{C:recover}
Let $G$ be an affine algebraic $k$-group. We have an isomorphism $G \cong \uAut^\otimes(\omega^G)$ of affine algebraic $k$-groups.
\end{cor}

\begin{cor}\label{C:fibre-iso}
Let $G$ be a smooth affine algebraic $k$-group. Let $\ell/k$ be a field extension and let $\eta\colon \Rep_k(G)\to \Vect_\ell$ be a fibre functor over $\ell$. Then $\uHom^\otimes (\omega^G,\eta)$ is a $G$-torsor over $\ell$. In particular, if $H^1(\ell,G)=\{1\}$ and $G(\ell)\neq \emptyset$, then $\omega^G$ is isomorphic to $\eta$ over $\ell$.
\end{cor}

\begin{proof}
Notice that we have an action
\[\uHom^\otimes (\omega^G,\eta) \times \uAut^\otimes (\omega^G) \longto \uHom^\otimes (\omega^G,\eta)\]
by pre-composition. By~\cite[Theorem 3.2 (i)]{tc}, $\uHom^\otimes (\omega^G,\eta)$ is an $\uAut^\otimes (\omega^G)$-torsor. In particular, it is a $G$-torsor over $\ell$ by Corollary~\ref{C:recover}.

Because $G$ is a $\ell$-group variety, $G$-torsors over $\eta$ are $\ell$-varieties by~\cite[Proposition 2.69]{lag}, whose isomorphism classes are classified by $H^1(\ell,G)$. It follows from the triviality of $H^1(\ell,G)$ that $\uHom^\otimes (\omega^G,\eta)(\ell)\cong G(\ell)$, hence $\uHom^\otimes (\omega^G,\eta)(\ell)\neq \emptyset$. \cite[Proposition 1.13]{tc} then implies the second assertion.
\end{proof}

To end this subsection, we give a Lie algebra version of Theorem~\ref{T:recover}. We start with recalling the notion of the Lie algebra of a $k$-group functor (see~\cite[II, \S 4]{DG} for a more details).

For any $R\in \Alg_k$, we define the \emph{$R$-algebra of dual numbers} $R[\varepsilon]\coloneqq R[X]/(X^2)$. Put $\varepsilon\coloneqq X+(X^2)$, we then have the canonical projection $\pi_R\colon R[\varepsilon]\to R,~\varepsilon\mapsto 0$.  Let $G$ be a $k$-group functor. We define
\[\Lie(G)(R) \coloneqq \Ker G(\pi_R).\]
Let $f\colon G\to H$ be a morphism of $k$-group functors. The commutative diagram
\begin{equation}\label{D:Lie}
\begin{tikzcd}
\Lie(G)(R)=\Ker(G(\pi_R)) \ar[d,"\iota_G"]  && \Lie(H)(R)=\Ker(H(\pi_R)) \ar[d,"\iota_H"]\\
G(R[\epsilon]) \ar[rr,"{f(R[\epsilon])}"]  \ar[d,"G(\pi_R)"]  &&  H(R[\epsilon])\ar[d,"H(\pi_R)"]\\
G(R)\ar[rr,"f(R)"]  &&  H(R)
\end{tikzcd}
\end{equation}
implies that $f(R[\epsilon])\circ \iota_G(X)\in \Lie(H)(R)$ for all $X\in \Lie(G)(R)$. We define $\Lie(f)\coloneqq f(R[\epsilon])\circ \iota_G\colon \Lie(G)(R) \to \Lie(H)(R)$. Hence, $\Lie(\func)(R)$ is functor from the category of $k$-group functors to that of abelian groups.

For an affine algebraic $k$-group $G$, we write $I$ for the kernel of the counit $\epsilon_G\colon k[G]\to k$. We have the following familiar group isomorphisms 
\[\fg \coloneqq \Lie(G)(k) \cong \Hom_k(I/I^2,k)\cong \Der_k(k[G],k).\]
Moreover, we have $\Lie(G)(R)\cong \fg_R$. The Lie bracket on $\Der_k(k[G],k)$ then gives a Lie bracket on $\fg_R$ and hence on $\Lie(G)(R)$. We will identify $\Lie(G)(R)$ and $\fg_R$, and call it the \emph{Lie algebra} of $G$ over $R$, whenever $G$ is affine algebraic. In this case, $\Lie(\func)(R)$ is a functor from the category of affine algebraic $k$-groups to that of Lie algebras over $R$.

\begin{rmk}\label{R:Lie-endo}
For any $d$-dimensional $G$-representation $(V,\rho_V)$, we write $\fgl_V\coloneqq \Lie(\GL_V)(k)$. We then have $\fgl_{V,R}=\{I_d+\varepsilon B \mid  B\in\Mat_{d,d}(R)\}$, after choosing a $k$-basis for $V$. Then $I_d+\varepsilon B \mapsto B$ gives a group isomorphism from $\fgl_{V,R}$ to $\End_R(V_R)$. Henceforth, we will identify $\Lie(\rho_V)(X)$ as an endomorphism of $V_R$, for all $X\in \fg_R$.

Replacing $H$ with $\GL_V$ and $f$ with $\rho_V$ in diagram~\eqref{D:Lie}, we obtain a morphism $\Lie(\rho_V)=\rho_V(R[\epsilon])\circ\iota_G \colon \fg_R \to \fgl_{V,R}$ of Lie algebras over $R$. Let $(W,\rho_W)\in \Rep_k(G)$, and let $\alpha\in\Hom_G(V,W)$. We then have $\alpha_R\circ\Lie(\rho_V)(X)=\Lie(\rho_W)(X)\circ\alpha_R$ for all $X\in\fg_R$.
\end{rmk}

Applying the functor $\Lie(\func)(R)$ on both sides of the isomorphism in Corollary~\ref{C:recover} then gives us an isomorphism $\fg_R\cong \Lie(\uAut^\otimes (\omega^G))(R)$ of Lie algebras over $R$. The following lemma indicates that the elements in $\Lie(\uAut^\otimes (\omega^G))(R)$ are exactly the derivatives (in the sense of taking derivations of conditions (i,ii,iii) in Theorem~\ref{T:recover}) of elements in $\uAut^\otimes (\omega^G)(R)$.

\begin{cor}\label{C:Lie-tannakian}
Let $G$ be an affine algebraic $k$-group and let $R$ be a $k$-algebra. Suppose that for any $(V,\rho_V)\in \Rep_k (G)$ we are given an $R$-linear endomorphism $\theta_V$ of $V_R$ subject to the conditions
\begin{itemize}
\item[(i)] $\theta_{V\stens W}=\theta_V\otimes \Id_{W_R} +\Id_{V_R}\otimes \theta_W$ for all $V,W\in \Rep_k(G)$;
\item[(ii)] $\theta_{\id}=0$ where $\id=k$ is the trivial $G$-representation;
\item[(iii)] $\theta_W\circ \alpha_R=\alpha_R\circ \theta_V$ for all $\alpha\in\Hom_G(V,W)$.
\end{itemize}
Then there exists a unique element $X\in\fg_R$ such that $\theta_V=\Lie(\rho_V)(X)$ for all $(V,\rho_V)\in \Rep_k(G)$.
\end{cor}

\begin{proof}  
For any $(V,\rho_V)\in \Rep_k(G)$ and $\theta_V\colon V_R\to V_R$, we define the following $R[\varepsilon]$-linear map
\[\varepsilon \theta_V \colon  V_{R[\varepsilon]} \longto  V_{R[\varepsilon]}, ~~~~~ v\otimes(x+y\varepsilon) \longmapsto \theta_V(v\otimes x)\varepsilon.\]
We then define the following $R[\varepsilon]$-linear endomorphism
\[\tilde \theta_V \coloneqq \Id_{V_{R[\varepsilon]}}+\varepsilon \theta_V \colon V_{R[\varepsilon]} \longto V_{R[\varepsilon]}. \]
Then $\tilde \theta_V\in \Lie(\GL_V)(R)\subseteq \GL_V(R[\varepsilon])$, because $\pi_R(\tilde \theta_V)=\Id_{V_R}$. 

We claim that the family 
\begin{equation}\label{E:Lie-tanakian}
\big\{\tilde \theta_V\colon V_{R[\varepsilon]}\to V_{R[\varepsilon]}  \mid (V,\rho_V)\in \Rep_k (G) \big\}   
\end{equation}
of $R[\varepsilon]$-linear endomorphisms satisfies conditions (i,ii,iii) in Theorem~\ref{T:recover}. Granting this claim for a moment, we then have that $\tilde\theta\in \uAut^\otimes(\omega^G)(R[\varepsilon])$. In particular, there exists a unique element $X\in G(R[\varepsilon])$ such that $\tilde\theta_V=\rho_V(X)$ for all $(V,\rho_V)\in\Rep_k(G)$. Since $\pi_R(\tilde\theta)=\Id\in \uAut^\otimes(\omega^G)(R)$, we have $\tilde\theta\in \Lie(\uAut^\otimes(\omega^G))(R)$. The isomorphism $\fg_R\cong \Lie(\uAut^\otimes (\omega^G))(R)$ then implies that $X\in \fg_R$. Furthermore, it follows from the construction that $\theta_V=\Lie(\rho_V)(X)$ for all $(V,\rho_V)\in\Rep_k(G)$, and the proposition follows.

It remains to prove the claim. Condition (ii) is clear from the construction. Given $(W,\rho_W)\in\Rep_k(G)$, we compute
\begin{align*}
\tilde \theta_{V\stens W} &=\Id_{(V\stens W)_R}+\varepsilon \theta_{V\stens W}\\
&=\Id_{(V\otimes W)_R}+\varepsilon(\theta_V\otimes \Id_{W_R} +\Id_{V_R}\otimes \theta_W)\\
&=(\Id_{V_R}+\varepsilon \theta_V)\otimes (\Id_{W_R}+\varepsilon \theta_W)\\
&=\tilde \theta_V\otimes \tilde \theta_W.
\end{align*}
Hence, \eqref{E:Lie-tanakian} satisfies condition (i). It remains to show that~~\ref{T:recover} satisfies condition (iii). Let $\alpha\in \Hom_G(V,W)$. For any $v\otimes(x+y\varepsilon)\in V_{R[\varepsilon]}$,  we compute
\begin{align*}
\alpha_{R[\varepsilon]}\circ \varepsilon \theta_V(v\otimes(x+y\varepsilon)) &=\alpha_{R[\varepsilon]} (\theta_V(v \otimes x))\varepsilon
=(\alpha_R\circ \theta_V)(v\otimes x )\varepsilon \\
&=(\theta_W\circ \alpha_R)(v\otimes x )\varepsilon 
=\theta_W (\alpha(v)\otimes x)\varepsilon
\\
&=\varepsilon\theta_W ( \alpha(v)\otimes (x+y\epsilon))
=\varepsilon\theta_W \circ \alpha_{R[\varepsilon]} (v\stens (x+y\varepsilon)).
\end{align*}
It follows that
\begin{align*}
\alpha_{R[\varepsilon]}\circ \tilde \theta_V&=\alpha_{R[\varepsilon]}\circ (\Id_{V_{R[\varepsilon]}}+\varepsilon \theta_V)
=\alpha_{R[\varepsilon]}+\alpha_{R[\varepsilon]}\circ \varepsilon \theta_V\\
&=\alpha_{R[\varepsilon]}+\varepsilon\theta_W \circ \alpha_{R[\varepsilon]}
=(\Id_{W_{R[\varepsilon]}}+\varepsilon \theta_W)\circ \alpha_{R[\varepsilon]}\\
&=\tilde \theta_W \circ \alpha_{R[\varepsilon]},
\end{align*}
as desired.

\end{proof}

\subsection{Filtered and graded fiber functors}\label{S:ziegler}
We recall the notion of filtered and graded fiber functors on tannakian categories following~\cite{ziegler}. Let $\Gamma$ be a totally ordered abelian group (written additively) and let $R\in \Alg_k$. A \emph{$\Gamma$-graded} $R$-module is an $R$-module $M$ together with a direct sum decomposition $M=\bigoplus\limits_{\gamma\in \Gamma} M_\gamma$.
A morphism between two $\Gamma$-graded $R$ modules $M$ and $N$ is an $R$-linear map $f\colon M\to N$ such that $f(M_\gamma)\subseteq N_\gamma$ for all $\gamma\in\Gamma$. We denote by $\Gamma$-$\Grad_R$ the category of $\Gamma$-graded modules over $R$. For $M,N\in \Gamma$-$\Grad_R$, we define the tensor product $(M\stens_R N)_\gamma= \bigoplus\limits_{\gamma'+\gamma''=\gamma} \big( M_{\gamma'}\stens_R N_{\gamma''}\big)$.

Let $M$ be an $R$-module. A \emph{$\Gamma$-filtration} on $M$ is an increasing map
\[\cF\colon \Gamma \longto \{R\text{-submodules of}~ M\},~~~\gamma \longmapsto \cF^\gamma M ,\]
such that $\cF^\gamma M=0$ for $\gamma\ll 0$ and $\cF^\gamma M=M$ for $\gamma\gg 0$, which is \emph{increasing} in the sense that $\cF^\gamma M \subseteq \cF^{\gamma'} M$ whenever $\gamma \leq \gamma'$. A \emph{$\Gamma$-filtered $R$-module} is an $R$-module $M$ with a $\Gamma$-filtration. To abbreviate notations, we sometimes denote $\cF^\gamma M$ by $M^\gamma$ if no confusion shall arise. A morphism between two $\Gamma$-filtered $R$-modules $M$ and $N$ is an $R$-linear map $f\colon M\to N$ such that $f(M^\gamma)\subseteq N^\gamma$ for all $\gamma\in \Gamma$. We denote by $\Gamma$-$\Fil_R$ the category of $\Gamma$-filtered modules over $R$.

Let $M$ be a $\Gamma$-filtered module over $R$. For any $\gamma\in \Gamma$, we put $\cF^{\gamma-} M\coloneqq \ssum\limits_{\gamma'<\gamma} \cF^{\gamma'} M$. We define
\[\gr_\cF^\gamma M\coloneqq \cF^\gamma M/\cF^{\gamma-}M.\]
Then $\gr_\cF M\coloneqq \bigoplus\limits_{\gamma\in \Gamma} \gr_\cF^\gamma M$ is a $\Gamma$-graded $R$ module, and is called the \emph{$\Gamma$-graded $R$-module associated to $\cF$}. We thus have a functor 
\[\gr\colon \Gamma\text{-}\Fil_R \longto \Gamma\text{-}\Grad_R.\]
Elements $\gamma\in \Gamma$ such that $\gr_\cF^\gamma M\neq 0$ are said to be the $\Gamma$-\emph{jumps} (or simply jumps) of $\cF$.

The tensor product structure in $\Gamma$-$\Fil_R$ is defined by
\[\cF^\gamma (M\stens_R N)=\ssum\limits_{\gamma'+\gamma''=\gamma} \cF^{\gamma'} M\stens_R \cF^{\gamma''}N,\]
for all $\Gamma$-filtered modules $M$ and $N$ over $R$.

A morphism $f\colon M\to N$ in $\Gamma$-$\Fil_R$ is said to be \emph{admissible} (or \emph{strict}) if 
\[f(M^\gamma)=f(M)\cap N^\gamma,~~~~~\forall \gamma\in \Gamma.\]
Following~\cite[\S 4.1]{ziegler}, we say that a short sequence 
$\begin{tikzcd}
0 \ar[r] &M' \ar[r,"f'"] & M  \ar[r,"f''"] &M''  \ar[r]  &0
\end{tikzcd}$
in $\Gamma$-$\Fil_R$ is \emph{exact} if  both of $f'$ and $f''$ are admissible, and the underlying short sequence in $\Mod_R$ is exact.

Let $\mathcal T$ be a tannakian category over $k$ and let $R$ be a $k$-algebra.
\begin{itemize}
\item[(i)] A \emph{$\Gamma$-graded fiber functor} on $\mathcal T$ over $R$ is an exact faithful $k$-linear tensor functor $\tau\colon \mathcal T \to \Gamma$-$\Grad_R$.
\item[(ii)] A \emph{$\Gamma$-filtered fiber functor} on $\mathcal T$ over $R$ is an exact faithful $k$-linear tensor functor $\eta\colon \mathcal T \to \Gamma$-$\Fil_R$.
\item[(iii)] Given an object $M=\bigoplus\limits_{\gamma\in\Gamma} M_\gamma$ in $\Gamma$-$\Grad_R$, we put $\cF^\gamma(M)\coloneqq\bigoplus\limits_{\gamma'\leq \gamma} M_{\gamma'}$. This gives rise to a functor $\fil\colon\Gamma$-$\Grad_R \to\Gamma$-$\Fil_R$.
\item[(iv)] A $\Gamma$-filtered  fiber functor $\eta$ is called \emph{splittable} if there exists a $\Gamma$-graded fiber functor  $\tau$ such that $\eta=\fil\circ \tau$, and $\tau$ is called a \emph{splitting} of $\eta$.
\end{itemize}

\begin{rmk}\label{R:filtered-fiber}
More concretely, a $\Gamma$-filtered fiber functor is a $k$-linear functor $\eta\colon \mathcal T \to \Gamma$-$\Fil_R$ satisfying the following properties (see \cite[Definition 4.2.6,~Remark 4.2.7]{dat}).
\begin{itemize}
\item[(i)] It is \emph{admissibly} (or \emph{strictly}) functorial, i.e., for any morphism $\alpha\colon X\to Y$ in $\mathcal T$, we have $\eta(\alpha)\big(\cF^\gamma  \eta(X)\big)=\eta(\alpha)(\eta(X))\cap\cF^\gamma \eta(Y)$ for all $\gamma\in\Gamma$.
\item[(ii)] It is compatible with tensor products, i.e., we have
\[ \cF^\gamma \big(\eta(X\stens Y)\big)=\ssum\limits_{\gamma'+\gamma''=\gamma} \cF^{\gamma'} \big(\eta(X)\big)\stens \cF^{\gamma''} \big(\eta(Y)\big),\]
for all $X,Y\in \Ob(\mathcal T)$ and $\gamma\in\Gamma$.
\item[(iii)] \[\cF^\gamma \eta(\id)=\left\{
\begin{array}{ll} 
R   &\text{for}~\gamma \geq 0\\
0    &\text{for}~ \gamma <0,
\end{array} \right.\]
where $\id$ is the identity object in $\mathcal T$. Note that $\cF^\gamma \eta(\id)$ is the identity object in $\Gamma$-$\Fil_R$.
\end{itemize}
\end{rmk}

\begin{cons}\label{C:Z-to-Q}
Let $(M,\cF)\in \ZZ$-$\Fil_R$ be a $\ZZ$-filtered module with $\ZZ$-jumps $\jmath_1< \cdots <\jmath_n$. For any $\gamma\in\Gamma\setminus \{0\}$, we define a $\Gamma$-filtered module $(M,[\gamma]_*\cF)$ by
\[([\gamma]_*\cF)^{x} M\coloneqq \left\{
\begin{array}{ll}
0       &\text{for}~x< \jmath_1\gamma \\
M^{\jmath_i}    ~~~ &\text{for}~ \jmath_i\gamma \leq x <\jmath_{i+1}\gamma,~ 1\leq i\leq n-1\\
M   &\text{for}~x\geq \jmath_n\gamma .
\end{array} \right.\]
We then have a fully faithful embedding $[\gamma]_*\colon \ZZ$-$\Fil_R \to \Gamma$-$\Fil_R$. Similarly, we have a fully faithful embedding $[\gamma]_*\colon \ZZ$-$\Grad_R \to \Gamma$-$\Grad_R$ by defining $[\gamma]_* \coloneqq \gr\circ [\gamma]_*\circ\fil$.

\end{cons}

To end this subsection, we exihibit the following theorem for later use. (Be aware that in \cite{ziegler}, the author only considers $\Gamma$-gradings and $\Gamma$-filtrations for $\Gamma=\ZZ$.)

\begin{thm}\cite[Theorem 4.15]{ziegler}\label{T:ziegler}
Let $\mathcal T$ be a tannakian category over a field $k$ and let $R$ be a $k$-algebra. Let $\eta\colon \mathcal T \to \ZZ$-$\Fil_R$ be a $\ZZ$-filtered fiber functor. If $\uAut^\otimes_R(\forg\circ \eta)$ is pro-smooth (i.e.\ a limit of smooth algebraic group schemes) over $R$, where $\forg\colon \ZZ$-$\Fil_R \to \Mod_R$ is the forgetful functor, then $\eta$ is splittable.
\end{thm}

\section{$G$-$\phi$-modules over the Robba ring}\label{S:G-phi}

We fix an affine algebraic $F$-group $G$ in this section.

\subsection{Definition}
Let $R\in\{\cE^\dagger,\cR,\tE^\dagger,\tR\}$ equipped with an absolute Frobenius lift $\phi$.

\begin{defn}\label{D:GIsoc}
A \emph{$G$-$\phi$-module} over $R$ is an exact faithful $F$-linear tensor functor
\[\I\colon \Rep_F(G) \longto \phimodR\]
which satisfies $\forg \circ \I = \omega^G \otimes R$, where $\forg \colon \phimodR \to \Mod_R$ is the forgetful functor. The category of $G$-$\phi$-modules over $R$ is denoted by $\GphimodR$, whose morphisms are morphisms of tensor functors.
\end{defn}

Let $(V,\rho)\in \Rep_F(G)$ and let $g\in G(R)$. We define $\I(g)(V)\coloneqq (V_R,g\phi)$, where
\begin{align*}
 g\phi\colon V_R  \longto V_R, \;\;\;\;\; v\otimes f  \longmapsto \rho(g)(v\otimes 1)\phi(f).
\end{align*}
Let $V,W\in\Rep_F(G)$. We have a canonical isomorphism $(V\otimes W)_\cR \cong V_\cR \otimes_\cR W_\cR$, and we will henceforth identify them. Given any $\alpha\in \Hom_G(V,W)$, we define $\I(g)(\alpha)\coloneqq \alpha_R$. We thus have the following $G$-$\phi$-module over $R$ (associated to $g$).
\[\I(g)\colon \Rep_F(G) \longto \phimodR,\;\;\;\;\; V\longmapsto (V,g\phi).\]
We call $\I(g)(V)=(V_R,g\phi)$ a \emph{$G$-$\phi$-module} over $R$ (associated to $g$).

For any $g\in G(R)$, we sometimes write $\Phi_g= \Phi_{g,V}$ for the $\phi$-linear action $g\phi$ on $V_R$. Both notations have their own advantages in practice.

\begin{rmk}\label{R:monoid}
For any $g\in G(R)$, we define $\phi(g)\coloneqq G(\phi)(g)$. For any $(V,\rho)\in\Rep_F(G)$, we have a commutative diagram
\[\begin{tikzcd}
G(R) \ar{r}{\rho(R)}  \ar{d}[swap]{G(\phi)}  &\GL_V(R) \ar{d}{\GL_V(\phi)}\\
G(R) \ar{r}[swap]{\rho(R)}  & \GL_V(R)
\end{tikzcd}\]
Hence $\rho(\phi(g))=\phi(\rho(g))$. For any $h\in G(R)$ and $n,m\geq 0$, we have the following formula in $G(R)\rtimes \la\phi\ra$
\[(h\phi^n)\circ (g\phi^m)=\big(h\phi^n(g)\big)\phi^{n+m}.\]

\end{rmk}

\subsection{The $\QQ$-filtered fiber functor $\HN_g$}\label{S:HN}

We fix an element $g\in G(\cR)$. 

\begin{cons}
For any $V\in\Rep_F(G)$, we have a $\phi$-module $(V_\cR,g\phi)$ over $\cR$. Kedlaya's slope filtration theorem~\cite[Theorem 6.10]{ked-ann} then provides a filtration
\[0\subseteq V_\cR^{\mu_1}\subseteq \cdots\subseteq V_\cR^{\mu_l}=V_\cR\]
satisfying 
\begin{itemize}
\item[$\bullet$] $V_\cR^{\mu_1}$ is pure of some slope $\mu_1\in\QQ$ and each $V_\cR^{\mu_i}/V_\cR^{\mu_{i-1}}$ is pure of some slope $\mu_i\in\QQ$ for $2\leq i\leq l$;
\item[$\bullet$] $\mu_1<\cdots<\mu_l$. 
\end{itemize}
We thus have an increasing map
\begin{align*}
\cHN_g \colon \QQ &\longto \{\cR\text{-submodules of}~V_\cR\}\\
x &\longmapsto \cHN_g^x(V_\cR),
\end{align*}
where
\[\cHN_g^x(V_\cR)=\left\{
\begin{array}{ll}
0       &\text{for}~x<\mu_1\\
V_\cR^{\mu_i}    ~~~ &\text{for}~ \mu_i\leq x <\mu_{i+1}, 1\leq i\leq l-1\\
V_\cR   &\text{for}~x\geq \mu_l.
\end{array} \right.\]
Then $(V_\cR,\cHN_g)$ is a $\QQ$-filtered module over $\cR$ with $\QQ$-jumps $\mu_1<\cdots<\mu_l$. We will denote $\cHN_g^x(V_\cR)$ by $V_\cR^x$ when $\cHN_g$ is clear in the context.
\end{cons}

\begin{thm}\label{T:filtered-fiber}
The assignments
\[V \longmapsto (V_\cR,\cHN_g) ~~~\text{and}~~~\alpha \longmapsto \alpha_\cR,\]
for all $\alpha\in \Hom_G(V,W)$, define a $\QQ$-filtered fiber functor
\[\HN_g \colon \Rep_F(G)  \longto \QQ\text{-}\Fil_\cR.\] 
\end{thm}
\begin{proof}
This is Proposition~\ref{P:filtered-fiber} and Proposition~\ref{P:exact} below.

\end{proof}

For any admissible extension $E$ of $K$, we first remark that the $\phi$-equivariant embedding $\psi\colon \cR\to \tR(E,t)$ is faithfully flat (see~\cite[Remark 3.5.3]{ked-relative}). We also remark that, if $M_1$ and $M_2$ are pure $\phi$-modules over $\cR$ of slopes $\mu_1$ and $\mu_2$, respectively, then $M_1\otimes_\cR M_2$ is pure of slope $\mu_1+\mu_2$ (see~\cite[Corollary 1.6.4]{ked-relative}). These facts will be repeatedly used in the sequel.

\begin{propo}\label{P:filtered-fiber}
The assignments in Theorem~\ref{T:filtered-fiber} yield a faithful $F$-linear tensor functor $\HN_g \colon\Rep_F(G) \to \QQ$-$\Fil_\cR$.
\end{propo}

\begin{proof}
Let $\id=F$ be the trivial $G$-representation. Then $\id\otimes_F \cR=\cR$ is of rank $1$ with slope $0$, proving that $\HN_g$ preserves identity objects.
 
We claim that $\HN_g$ is functorial. Let $\alpha\in \Hom_G(V,W)$ be a morphism of finite-dimensional $G$-modules. We need to show that $\alpha_\cR(V_\cR^x)\subseteq W_\cR^x$ for all $x\in \QQ$. Choose by Lemma~\ref{L:admissible} an admissible extension $E$ of $K$ such that $\kk_E$ is strongly difference-closed. For any fixed $x\in\QQ$, we set $V_{\tR(E,t)}^x \coloneqq V_\cR^x\bigotimes_\cR {\tR(E,t)}$, and $W_{\tR(E,t)}^x \coloneqq W_\cR^x\bigotimes_\cR {\tR(E,t)}$. By Lemma~\ref{L:tilde-tensor} (iv), we have a decomposition $W_{\tR(E,t)}=W_{\tR(E,t)}^x \soplus W_{\tR(E,t)}'$ of $\phi$-modules over $\tR(E,t)$, where $W_{\tR(E,t)}^x$ (resp.\ $W_{\tR(E,t)}'$) has slopes less or equal to $x$ (resp.\ greater than $x$). By Lemma~\ref{L:nonzero-morphism}, the induced morphism $V_{\tR(E,t)}^x \to W_{\tR(E,t)}'$ of $\phi$-modules is zero. We thus have $\alpha_{\tR(E,t)} \big(V_{\tR(E,t)}^x \big)\subseteq W_{\tR(E,t)}^x$. Given any $\bv\in V_\cR^x$, we may write $\alpha_{\tR(E,t)}(\bv\otimes 1)=\alpha_\cR(\bv)\otimes 1=\bsum\limits_{i\in I} \bw_i\otimes s_i$ for some finite set $I$, with $\bw_i\in W_{\cR}^x$ and $s_i\in {\tR(E,t)}$ for all $i\in I$. Let $M$ be the $\cR$-submodule of $W_{\cR}$ generated by $\alpha_\cR(\bv)$ and the $\bw_i$, and let $N$ be the $\cR$-submodule of $W_{\cR}^x$ generated by the $\bw_i$. We then have $(M/N)\bigotimes_\cR \tR(E,t) \cong (M\bigotimes_\cR \tR(E,t))/ (N\bigotimes_\cR \tR(E,t))=0$. It follows that $M/N=0$ as $\cR\to \tR(E,t)$ is faithfully flat. We thus have $\alpha_\cR(\bv)\in N\subseteq W_{\cR}^x$, as desired.

It remains to show that $\HN_g$ preserves tensor products (in the sense of Remark~\ref{R:filtered-fiber} (ii)). Let $V$ and $W$ be two finite-dimensional $G$-modules, and suppose that the slope filtration of $(V_\cR,g\phi)$ (resp.\ $(W_\cR,g\phi)$) has jumps $\mu_1<\cdots<\mu_{l_V}$  (resp.\ $\nu_1<\cdots<\nu_{l_W}$). By Lemma~\cite[Lemma 16.4.3]{pde}, $\big((V\bigotimes_F W)_{\cR},g\phi \big)$ has jumps $\{\mu_i+\nu_j \mid 1\leq i\leq {l_V}, 1\leq j\leq {l_W}\}$. Fix any $1\leq l\leq {l_V}$ and $1\leq s\leq {l_W}$, we need to show
\begin{equation}\label{E:tensor}
 (V \stens_F W)_{\cR}^{\mu_{l}+\nu_{s}}= \ssum\limits_{x,y\in\QQ\atop\\ x+y=\mu_{l}+\nu_{s}} V_\cR^x \stens_\cR W_\cR^y,
\end{equation}
and we will do so in the remainder of the proof.

We claim that
\[\ssum\limits_{x,y\in\QQ\atop\\ x+y=\mu_{l}+\nu_{s}} V_\cR^x \stens_\cR W_\cR^y= \ssum\limits_{\mu_i+\nu_j\leq \mu_l+\nu_s\atop\\ 1\leq i\leq {l_V}, 1\leq j\leq {l_W}}V_\cR^{\mu_i} \stens_\cR W_\cR^{\nu_j}.\]
It is clear that the RHS is contained in the LHS, we now show the reverse containment. Let $x,y\in\QQ$ such that $x+y=\mu_{l}+\nu_{s}$. If $x<\mu_1$ or $y<\nu_1$, then $V_\cR^x\otimes_\cR W_\cR^y=0$ which is contained in the RHS. Otherwise, there exists the largest integer $1\leq i\leq {l_V}$ (resp.\ $1\leq j\leq {l_W}$) with the property that $\mu_i\leq x$ (resp.\ $\nu_j\leq y$). We then have $V_\cR^x\bigotimes_\cR W_\cR^y=V_\cR^{\mu_i}\bigotimes_\cR W_\cR^{\nu_j}$ and $\mu_i+\nu_j\leq \mu_{l}+\nu_{s}$. The claim is thus proved.

From Lemma~\ref{L:tilde-tensor} (iv), we see that
\[ \big(V \stens_F W \big)_{\tR(E,t)}^{\mu_{l}+\nu_{s}}= \Big( \ssum\limits_{\mu_i+\nu_j\leq \mu_{l}+\nu_{s}\atop\\ 1\leq i\leq {l_V}, 1\leq j\leq {l_W}}V_\cR^{\mu_i}\stens_\cR W_\cR^{\nu_j} \Big)\stens_\cR {\tR(E,t)}.\]
Therefore, we have 
\[ \big(V \stens  W \big)_{\cR}^{\mu_{l}+\nu_{s}}=\ssum\limits_{\mu_i+\nu_j\leq \mu_{l}+\nu_{s}\atop\\ 1\leq i\leq {l_V}, 1\leq j\leq {l_W}} V_\cR^{\mu_i}\stens_\cR W_\cR^{\nu_j}\]
by Lemma~\ref{L:tilde-tensor} (ii) and the fact that $\cR \to {\tR(E,t)}$ is faithfully flat. The desired equality~\eqref{E:tensor} then follows from the preceding claim.

\end{proof}

Let $(M,\Phi)$ be a $\phi$-module over $\tR$ of rank $d$. Then $\Phi$ is invertible since $\tR$ is inversive, and $(M,\Phi^{-1})$ is a $\phi^{-1}$-module over $\tR$. More explicitly, let $A\in\GL_d(\tR)$ be the matrix of action of $\Phi$ in some basis for $M$ over $\tR$. Then in the same basis, the matrix of action of $\Phi^{-1}$ is $\phi^{-1}(A^{-1})$. For example, if $M=V_{\tR}$ for some $V\in \Rep_F(G)$ and $\Phi=\psi(g)\phi$, then

\[\big(\psi(g)\phi \big) \cdot \big(\phi^{-1}(\psi(g^{-1}))\phi^{-1}\big)=1\]
in $G(\tR)\rtimes \la\phi\ra$ (cf.\ Remark~\ref{R:monoid}), which implies that $\Phi^{-1}=\phi^{-1}(\psi(g^{-1}))\phi^{-1}$.

Let $M$ be a standard $\phi$-module over $\tR$ of slope $\mu=s/r$ with $r>0$ and $(s,r)=1$. Namely, we have a standard basis $e_1,\cdots,e_r$ in which $\Phi$ acts via 
\[A=
\begin{pmatrix}\footnotesize
0  &  &  &\pi^s\\
1 &\ddots &  & \\
&\ddots  &\ddots &\\
&& 1 &0
\end{pmatrix}.
\]
Then 
\[\phi^{-1}(A^{-1})=
\begin{pmatrix}\footnotesize
0  &1 & \\
 &\ddots   &\ddots & \\
 & &\ddots &1\\
\pi^{-s} &  &    &0
\end{pmatrix},
\]
which implies that $(M,\Phi^{-1})$ is a standard $\phi^{-1}$-module pure of slope $-\mu$.

\begin{propo}\label{P:exact}
The functor $\HN_g\colon \Rep_F(G) \to \QQ\text{-}\Fil_\cR$ is exact.
\end{propo}
\begin{proof}
Let $\alpha\in \Hom_G(V,W)$ be a morphism of finite-dimensional $G$-modules. We need to show that $\alpha_\cR(V_\cR^x)=\alpha_\cR(V_\cR)\cap W_\cR^x$ for all $x\in\QQ$. For any fixed $x\in\QQ$, the functoriality in Proposition~\ref{P:filtered-fiber} already implies that $\alpha_\cR(V_\cR^x)\subseteq \alpha_\cR(V_\cR)\cap W_\cR^x$. Thus, it suffices to show that for any non-zero element $\bv\in V_\cR$ such that $\alpha_\cR(\bv)\in W_\cR^x$, there exists $\bv'\in V_\cR^x$ with $\alpha_\cR(\bv)=\alpha_\cR(\bv')$. 

By Lemma~\ref{L:tilde-tensor} (iv), we have decompositions
\begin{equation}\label{E:tilde-decomposition}
V_{\tR(E,t)}=V_{\tR(E,t)}^x \soplus V_{\tR(E,t)}'~~~\text{and}~~~W_{\tR(E,t)}=W_{\tR(E,t)}^x \soplus W_{\tR(E,t)}'
\end{equation}
of $\phi$-modules over $\tR(E,t)$, in which $V_{\tR(E,t)}^x$ and $W_{\tR(E,t)}^x$ have slopes less or equal to $x$, while $V_{\tR(E,t)}'$ and $W_{\tR(E,t)}'$ have slopes greater than $x$. Notice that the composition
\[\begin{tikzcd}
\xi\colon V_{\tR(E,t)}' \ar[r] & V_{\tR(E,t)}^x \soplus V_{\tR(E,t)}' \ar[r,"\alpha_{\tR(E,t)}"] &W_{\tR(E,t)}^x \soplus W_{\tR(E,t)}' \ar[r] &W_{\tR(E,t)}^x
\end{tikzcd}\]
is a morphism of $\phi$-modules. We claim that $\xi=0$. We write $\Phi=\psi(g)\phi$, then $\Phi^{-1}=\phi^{-1}(\psi(g^{-1}))\phi^{-1}$. Since $\alpha$ is $G$-equivariant and $\phi^{-1}(\psi(g^{-1}))\in G(\tR(E,t))$, we have that $\alpha_{\tR}\colon (V_{\tR(E,t)},\Phi^{-1}) \to (W_{\tR(E,t)},\Phi^{-1})$ is a morphism of $\phi^{-1}$-modules. On the other hand, we also have decompositions of $\phi^{-1}$-modules as in~\eqref{E:tilde-decomposition}, together with the induced morphism $\xi\colon V_{\tR(E,t)}' \to W_{\tR(E,t)}^x$ of $\phi^{-1}$-modules. But in this case, $V_{\tR(E,t)}'$ has slopes less than $x$, while $W_{\tR(E,t)}^x$ has slopes greater or equal to $x$. It then follows from Lemma~\ref{L:nonzero-morphism} that $\xi=0$, as claimed.

Therefore, we find $\bv_1,\cdots,\bv_n \in V_\cR^x$ and $s_1,\cdots,s_n\in \tR(E,t)$ such that
\[ \alpha_{\tR(E,t)}(\bv\otimes 1)=\alpha_\cR(\bv)\otimes 1 =\ssum\limits_{i=1}^n \alpha_\cR(\bv_i)\otimes s_i.\]
Let $M$ be the submodule of $W_{\cR}$ generated by $\alpha_\cR(\bv)$ and the $\alpha_\cR(\bv_i)$, and let $N$ be the submodule generated by the $\alpha_\cR(\bv_i)$. We then have
\[(M/N)\stens_\cR \tR(E,t) \cong (M\stens_\cR \tR(E,t))/ (N\stens_\cR \tR(E,t))=0.\] 
It follows that $M/N=0$ as $\cR\to \tR(E,t)$ is faithfully flat, and hence $\alpha_\cR(\bv)=\sum\limits_{i=1}^n  r_i\alpha_\cR(\bv_i)\in W_\cR^x$ for some $r_i\in \cR$. Put $\bv'\coloneqq \sum\limits_{i=1}^n r_i\bv_i \in V_\cR^x$, we then have $\alpha_\cR(\bv')=\alpha_\cR(\bv)$, as desired.

\end{proof}

\subsection{Splittings of $\HN_g$}\label{S:splittings}

As before, we fix an element $g\in G(\cR)$. In \S~\ref{S:HN}, we have constructed a $\QQ$-filtered fiber functor $\HN_g\colon \Rep_F(G) \to \QQ$-$\Fil_\cR$. In this subsection, we show that $\HN_g$ is splittable whenever $G$ is smooth. Our strategy goes as follows. We first use Lemma~\ref{L:Q-to-Z} reducing $\HN_g$ to a $\ZZ$-filtered fiber functor $\HN_g^\ZZ$ to which Theorem~\ref{T:ziegler} is applicable. This $\HN_g^\ZZ$ then admits a $\ZZ$-splitting. Finally, in Theorem~\ref{T:Q-splitting}, we lift such a $\ZZ$-splitting to a $\QQ$-splitting of $\HN_g$.

\begin{defn}
We define the \emph{support} of $\HN_g$ by
\[\Supp(\HN_g) \coloneqq \{x\in \QQ \mid \gr^x_{\HN_g}(V)\neq 0~\text{for some}~ V\in \Rep_F(G)\}.\]
\end{defn}

Notice that $\Supp(\HN_g)$ is the set of jumps of the slope filtrations of $(V_\cR,g\phi)$ for all $V\in \Rep_F(G)$.

The general idea of the following construction was addressed in~\cite{ans}, after Definition 2.5 in loc.\ cit.; we will make it more explicit in our case.

\begin{cons}\label{C:d_g}
Let $W\in \Rep_F(G)$ be a faithful representation. Since $G$ is algebraic, $W$ is a tensor generator for $\Rep_F(G)$, i.e., any representation $V$ of $G$ is a subquotient of a direct sum of representations $\bigotimes^m (W \bigoplus W^\vee)$ for various $m\in \NN$. (See~\cite[Theorem 4.14]{lag}.) Therefore, $\Supp(\HN_g)$ is the additive subgroup of $\QQ$ finitely generated by the $\QQ$-jumps $\nu_1,\cdots,\nu_n$ of $(W_\cR,g\phi)$. We write $\nu_i=s_i/d_i$ with $d_i>0$ and $(s_i,d_i)=1$ for $1\leq i\leq n$. Let $d_g\in\NN$ be the least common multiple of the $d_i$. We then have $d_g\nu_i\in\ZZ$ for $1\leq i\leq n$.
In particular, we have
\[d_g=\min \{m\in\NN \mid mx\in \ZZ, \forall x\in \Supp(\HN_g)\}.\]
Therefore, $d_g$ is uniquely determined by $g$. We call $d_g$ the \emph{least common denominator} of $g$.

\end{cons}

\begin{rmk}
We conclude from Construction~\ref{C:d_g} that $\Supp(\HN_g)$ is isomorphic to $\ZZ$ or $0$. In fact, suppose that $d_g\nu_1,\cdots,d_g\nu_n$ are not all zero, we then let $D$ be the greatest common divisor of the non-zero ones. Otherwise, we let $D=0$. We then have that $d_g\cdot\Supp(\HN_g)=D\ZZ$. Hence,
\[ \Supp(\HN_g) \cong
\left\{\begin{array}{ll}
\ZZ    &\text{for}~  D\neq 0\\
0      &\text{for}~  D=0.
\end{array}\right.\]

\end{rmk}

\begin{lem}\label{L:Q-to-Z}
$\HN_g$ factors through a $\ZZ$-filtered fiber functor $\HN^{\ZZ}_g\colon \Rep_F(G) \to \ZZ$-$\Fil_\cR$ which makes the diagram
\[\begin{tikzcd}[scale=1.2]
\Rep_F(G) \ar[rr,"\HN_g"]\ar[d,swap,"\HN^{\ZZ}_g"]  && \QQ\text{-}\Fil_\cR\\
\ZZ\text{-}\Fil_\cR \ar[urr,,swap,"{[d_g^{-1}]}_*"]
\end{tikzcd}\]
commute.
\end{lem}

We remark that the functor ${[d_g^{-1}]}_*$ (cf.\ Construction~\ref{C:Z-to-Q}) is nothing but relabelling the jumps by multiplying all jumps with $d_g^{-1}$. In particular, this lemma implies that $\gr^x_{\HN_g}(V)=\gr^{d_gx}_{\HN^\ZZ_g}(V)$ for all $x\in\QQ$ and $V\in \Rep_F(G)$.

\begin{proof}[Proof of Lemma~\ref{L:Q-to-Z}]
Let $V\in\Rep_F(G)$ and let $\mu_1,\cdots,\mu_l$ be the $\QQ$-jumps of $(V_\cR,g\phi)$. We then have $d_g\mu_i\in\ZZ$ for all $i$. We have an increasing map
\begin{align*}
\cF_g \colon \ZZ &\longto \{\cR\text{-submodules of}~V_\cR\}\\
x &\longmapsto \cF_g^x(V_\cR),
\end{align*}
where
\[\cF_g^x(V_\cR)\coloneqq \left\{
\begin{array}{ll}
0       &\text{for}~x<d_g\mu_1\\
\cHN_g^{\mu_i}(V_\cR)    ~~~ &\text{for}~ d_g\mu_i\leq x <d_g\mu_{i+1}, 1\leq i\leq l-1\\
V_\cR   &\text{for}~x\geq d_g\mu_l.
\end{array} \right.\]
Then $(V_\cR,\cF_g)$ is a $\ZZ$-filtered module over $\cR$ with $\ZZ$-jumps $d_g\mu_1<\cdots<d_g\mu_l$. We thus have a $\ZZ$-filtered fiber functor
\begin{align*}
\HN^{\ZZ}_g \colon \Rep_F(G) &\longto \ZZ\text{-}\Fil_\cR\\
V  &\longmapsto (V_\cR,\cF_g),
\end{align*}
satisfying $\HN_g=[d_g^{-1}]_*\circ \HN_g^\ZZ$.
\end{proof}

By the definition of $\uAut^\otimes$ and Corollary~\ref{C:recover}, we have $\uAut^\otimes(\omega^G)(R)=\Aut^\otimes (\omega^G_R)\cong G(R)$ for all $R\in \Alg_k$. For any $R$-algebra $S$, we then have
\[\uAut^\otimes(\omega_R^G)(S)=\Aut^\otimes (\omega^G_R\otimes S)=\Aut^\otimes (\omega^G_S)\cong G_R(S).\]

\begin{propo}\label{P:Z-splitting}
Let $G$ be a smooth $F$-group. Then $\HN^{\ZZ}_g$ is splittable.
\end{propo}
\begin{proof}
Since $\forg\circ \HN^{\ZZ}_g=\omega^G\otimes \cR$, we have
\[\uAut^\otimes_\cR(\forg\circ \HN^{\ZZ}_g)= \uAut^\otimes_\cR(\omega^G_ \cR)\cong G_\cR.\]
Notice that $G_\cR$ is smooth over $\cR$, the proposition then follows from Theorem~\ref{T:ziegler}.

\end{proof}

\begin{thm}\label{T:Q-splitting}
Let $G$ be a smooth $F$-group. Then the $\QQ$-filtered fiber functor $\HN_g$ is splittable.
\end{thm}
\begin{proof}
Choose a splitting $\tau_g\colon \Rep_F(G)\to \ZZ$-$\Grad_\cR$ of $\HN^{\ZZ}_g$ by Proposition~\ref{P:Z-splitting}, we then have a $\QQ$-graded fiber functor ${[d_g^{-1}]}_*\circ \tau_g\colon \Rep_F(G)\to \QQ$-$\Grad_\cR$. On the other hand, we have the diagram
\begin{equation}\label{D:splittings}
\begin{tikzcd}
&& \Rep_F(G)   \ar[drrr,"\HN_g"]\ar[dll,swap,"\tau_g"]\ar[d,swap,"\HN^{\ZZ}_g"]  \\
\ZZ\text{-}\Grad_\cR \ar[rr,"\fil"] \ar[drr,swap,"{[d_g^{-1}]}_*"] &&\ZZ\text{-}\Fil_\cR \ar[rrr,"{[d_g^{-1}]}_*"]  &&& \QQ\text{-}\Fil_\cR\\
 && \QQ\text{-}\Grad_\cR \ar[urrr,swap,"\fil"]
\end{tikzcd}
\end{equation}
with the upper-left, the upper-right and the bottom triangles commutative. Here, the commutativity of the upper-left (resp.\ upper-right ) triangle follows from Proposition~\ref{P:Z-splitting} (resp.\ Lemma~\ref{L:Q-to-Z}); for the bottom one, we note that ${[d_g^{-1}]}_*\circ \fil=\fil\circ {[d_g^{-1}]}_*$. Hence, the outer diagram also commutes, which implies that $\HN_g$ factors through the $\QQ$-graded fiber functor ${[d_g^{-1}]}_*\circ \tau_g$, as desired.

\end{proof}

\subsection{The slope morphism}\label{S:slope-morphism}

Let $R$ be a commutative ring with $1$ and let $\Gamma$ be an abelian group (not necessarily finitely generated). We first continue the discussions in \S~\ref{S:ziegler} to see how $\Gamma$-gradings over $R$ are related to $D(\Gamma)$-modules, for some affine group scheme $D(\Gamma)$ which will be defined as follows.

The group algebra $R[\Gamma]\coloneqq \bigoplus\limits_{\gamma\in \Gamma} Re_\gamma$ carries a Hopf algebra structure, where the comultiplication is given by $\Delta(e_\gamma)=e_\gamma\otimes e_\gamma$, the counit is given by $\epsilon(e_\gamma)=1$, and the antipode is given by $S(e_\gamma)=e_{-\gamma}$, for all $\gamma\in \Gamma$. We denote by $D_R(\Gamma)$ (or simply $D(\Gamma)$ when $R$ is clear in the context) the affine $R$-group scheme represented by $R[\Gamma]$.

\begin{lem} \cite[Proposition II. 2.5]{DG}\label{L:D}
$\Gamma$-$\Grad_R$ is equivalent to the category of $D(\Gamma)$-modules.
\end{lem}

\begin{cor}\label{C:D-to-G_m}
For any $\gamma\in\QQ$, the functor $[\gamma]_*\colon \ZZ\text{-}\Grad_R \to \QQ\text{-}\Grad_R$ corresponds to the character $\chi_\gamma\colon \DD_R\to \GG_{m,R}$.
\end{cor}
\begin{proof}
Let $M\in \ZZ\text{-}\Grad_R$. By Lemma~\ref{L:D}, we may write $M=\bigoplus\limits_{n\in\ZZ} M_n$ as a direct sum of eigenmodules. By construction, we have $[\gamma]_*(M)=\bigoplus\limits_{n\in\ZZ} ([\gamma]_*(M))_{\gamma n}$ with $([\gamma]_*(M))_{\gamma n}=M_n$ for all $n$, which is also a decomposition into eigenmodules. Therefore, giving $[\gamma]_*$ is equivalent to giving the commutative diagram
\[\begin{tikzcd}
M_n   \ar[r,equal] \ar[d]  &   ([\gamma]_*(M))_{\gamma n}\ar[d]\\
M_n\stens_R R[\ZZ]    \ar[r]    &([\gamma]_*(M))_{\gamma n}\stens_R R[\QQ]
\end{tikzcd}\]
of $R$-modules for all $n\in\ZZ$ such that $M_n\neq 0$. Here, the left (resp.\ the right) vertical arrow is given by $m\mapsto m\otimes e_n$ (resp.\ $m\mapsto m\otimes e_{\gamma n}$). The diagram then corresponds to $R[\ZZ]\to R[\QQ],~e_1\mapsto e_\gamma$, as desired.
\end{proof}

We now apply the preceding discussions to the functors constructed in \S~\ref{S:splittings}, following~\cite[4]{kott1}.

\begin{cons}\label{C:slope-morphism}
Let $g\in G(\cR)$, we fix a splitting $\tau_g$ of $\HN_g^\ZZ$ given by Proposition~\ref{P:Z-splitting}. For any $(V,\rho)\in\Rep_F(G)$, $\tau_g$ gives a decomposition of $V_\cR$, which induces a morphism $\lambda_{\rho,g}\colon \GG_{m,\cR} \to \GL_{V,\cR}$ by Lemma~\ref{L:D}. Let $S$ be an $\cR$-algebra and let $a\in \GG_{m,\cR}(S)$. We then have a family
\[\big\{\lambda_{\rho,g}(a)\colon V_S\to V_S \mid (V,\rho)\in\Rep_F(G) \big\}\]
$  $of $S$-linear maps. Because $\tau_g$ is a tensor functor, this family satisfies conditions (i,ii,iii) in Theorem~\ref{T:recover}. We thus find a unique element $b\in G_\cR(S)$ such that $\lambda_{\rho,g}(a)=\rho(b)$ for all $(V,\rho)\in\Rep_F(G)$. The assignment $a\mapsto b$ is functorial in $S$ since both $\lambda_{\rho,g}$ and $\rho$ are functorial.  We then have a morphism of $\cR$-groups
\[\lambda_g\colon \GG_{m,\cR}  \longto G_\cR,\]
which is said to be the \emph{$\ZZ$-slope morphism} of $g$. 

By Corollary~\ref{C:D-to-G_m}, ${[d_g^{-1}]}_*$ gives a unique morphism $\chi_{d_g^{-1}}\colon \DD_\cR \to \GG_{m,\cR}$. We define 
\[\nu_g\coloneqq \lambda_g\circ \chi_{d_g^{-1}} \colon \DD_\cR  \longto G_\cR,\]
which is said to be the \emph{$\QQ$-slope morphism} of $g$.
\end{cons}

The following example demonstrates explicitly how $\lambda_g$ and $\nu_g$ are related to the splittings constructed in \S~\ref{S:splittings} (cf.\ Diagram~\ref{D:splittings}).

\begin{exmp}
Let $(V,\rho)\in\Rep_F(G)$ and suppose that the slope filtration of $(V_\cR,g\phi)$ is
\[0\subseteq V_\cR^{\mu_1} \subseteq \cdots \subseteq V_\cR^{\mu_l}=V_\cR\]
with jumps $\mu_1<\cdots <\mu_l$. Let
\begin{equation}\label{E:decomp}
\textstyle  V_\cR=V_{\cR,\mu_1}\soplus\cdots \soplus V_{\cR,\mu_l}
\end{equation}
be a splitting of $\HN_g(V)$, i.e., we have $\bigoplus\limits_{i=1}^j V_{\cR,\mu_i}=V_\cR^{\mu_j}$ for all $1\leq j\leq l$. 

First, we fix $1\leq i\leq l$. Let $S\in \Alg_\cR$ and $a\in \DD_\cR(S)$, then $\rho\circ \nu_g(a)$ acts on $V_{\cR,\mu_i}\stens_\cR S$ via multiplication by $\chi_1(a)^{\mu_i}$. On the other hand, $\tau_g$ induces the same decomposition~\eqref{E:decomp} of $V_\cR$. Furthermore, $\rho\circ\lambda_g(b)$ acts on $V_{\cR,\mu_i}$ via multiplication by $b^{d_g\mu_i}$, for all $b\in \GG_{m,\cR}(S)$. Then on $V_{\cR,\mu_i}\stens_\cR S$, we have
\[\rho\circ \nu_g(a)=\chi_1(a)^{\mu_i}=\big(\chi_{d_g^{-1}}(a)^{d_g}\big)^{\mu_i}=\rho\circ \lambda_g\big( \chi_{d_g^{-1}}(a)\big)=\rho\circ \lambda_g\circ \chi_{d_g^{-1}}(a)\]
We next apply this result to all $1\leq i\leq l$. Since $V_\cR=\bigoplus\limits_{i=1}^l V_{\cR,\mu_i}$, we conclude that 
$\rho\circ \nu_g=\rho\circ \lambda_g\circ \chi_{d_g^{-1}}$. It follows that $\nu_g=\lambda_g\circ \chi_{d_g^{-1}}$ once we choose a faithful representation, as is expected from the definition of $\nu_g$.
\end{exmp}

If $G=\GL_V$ for some $V\in \Vect_F$, we consider the standard representation $(V,\rho)$ of $G$ where $\rho$ is the identity. The discussion in the above example then indicates that the image of $\lambda_g$ is contained in a split maximal torus in $G_\cR$; we conjecture that this property holds true for an arbitrary split reductive $F$-group $G$, and we shall give one more evidence as follows.

\begin{exmp}\label{E:SL_V}
Fix a $d$-dimensional $F$-vector space $V$. For any $R\in\Alg_F$, we define
\[\SL_V(R) \coloneqq \{ g\in \Aut_R(V_R)=\GL_V(R) \mid \det(g)=1\} .\]
The affine algebraic $F$-group $\SL_V$ is called the \emph{special linear group} (associated to $V$). 

Fix an arbitrary $g\in \SL_V(\cR)$. With the notation as in Construction~\ref{C:Phi'}, we suppose the jumps of the slope filtration of $(V_\cR,\Phi_g)$ are $\mu_1,\cdots,\mu_l$ and $\xi_g(V)=\bigoplus\limits_{i=1}^l V_{\cR,\mu_i}$. For each $i$, we write $r_i=\rk_\cR (V_{\cR,\mu_i})$, then the $r_i$-th exterior product $\Lambda^{r_i}(V_{\cR,\mu_i})$ is of rank $1$. We choose a generator $m_i$, then $\Lambda^{r_i}(\Phi_{g,\mu_i})(m_i)=f_i m_i$ for some $f_i\in \cR^\times=(\cE^\dagger)^\times$. Let $\nu$ be the valuation of the $1$-Gauss norm on $\cE^\dagger$. We then have $\mu_i=\frac{\nu(f_i)}{r_i}$ by~\cite[Definition 1.4.4]{ked-relative}. 

Let $e_1,\cdots,e_d$ be a basis for $V$ over $F$, and let $A\in \SL_d(\cR)$ be the matrix of action of $\Phi_g$ in $e_1\otimes 1,\cdots, e_d\otimes 1$. Let $U\in \GL_d(\cR)$ represent a change-of-basis over $\cR$. Then in the new basis, the matrix of action of $\Phi_g$ is $U^{-1}A \phi(U)$. Notice that $\det(U)\in (\cE^\dagger)^\times$ and $\phi$ preserves $\nu$, we then have
\[\nu\big( \det(U^{-1}A \phi(U) )\big)= \nu \big(\det(U^{-1}) \det(A) \phi(\det(U))\big)=\nu(\det(A)),  \]
which implies that the valuation of the determinant of the matrix of action of $\Phi_g$ is invariant under change-of-basis. We denote by $\nu(\det (\Phi_g))$ this invariant. In particular, we have $\nu(\det (\Phi_g))=0$ since $\det(A)=1$ by assumption. We thus have
\[0=\nu(\det (\Phi_g))=\nu(\det (\Phi'_g))=\nu(f_1)+\cdots +\nu(f_l)=r_1\mu_1+\cdots +r_l\mu_l.\]

Let $S\in \Alg_\cR$ and $t\in \GG_{m,\cR}(S)$. Since $\lambda_g(t)$ acts on each $V_{\cR,\mu_i}\stens_\cR S$ via multiplication by $t^{d_g\mu_i}$ where $d_g$ is the least common denominator of $g$, we then have
\[\det(\lambda_g(t))= t^{d_g(r_1\mu_1+\cdots  r_l \mu_l)}=1.\]
Therefore, the image of $\lambda_g$ is contained in a split maximal torus in $\SL_{V,\cR}$.
\end{exmp}

\section{$G$-$(\phi,\nabla)$-modules over the Robba ring}\label{C:G-phi-nabla}

In this section, we fix an affine algebraic group $F$-group $G$.

\subsection{Definition and an identification}\label{S:G-phi-nabla}

Let $R\in\{\cE^\dagger,\cR,\tE^\dagger,\tR\}$ equipped with an absolute Frobenius lift $\phi$.

\begin{defn}\label{D:GIsoc}
A \emph{$G$-$(\phi,\nabla)$-module} over $R$ is an exact faithful $F$-linear tensor functor
\[\I\colon \Rep_F(G) \longto \mathbf{Mod}^{\phi,\nabla}_R\]
which satisfies $\forg \circ \I = \omega^G \otimes R$, where $\forg \colon \phimodR \to \Mod_R$ is the forgetful functor. The category of $G$-$(\phi,\nabla)$-modules over $R$ is denoted by $\GphinabR$, whose morphisms are morphisms of tensor functors. A $G$-$(\phi,\nabla)$-module $\I$ over $R$ is called \emph{unit-root} if $\I(V)$ is a unit-root $(\phi,\nabla)$-module over $R$ for all $V\in \Rep_F(G)$.
\end{defn}

We put $\partial=\partial_t=d/dt$, the usual derivation on $R$. We also put
\[\muu=\muu (\phi,t)\coloneqq \partial(\phi(t)).\]
Let $\Omega_R^1=\Omega_{R/K}^1$ be the free $R$-module generated by the symbol $dt$, with the $K$-linear derivation $d\colon R \longto \Omega^1_R,~~~ f\longmapsto  \partial(f) dt$. We also define a $\phi$-linear endomorphism
\begin{align*}
d\phi\colon  \Omega^1_R \longto  \Omega^1_R,\;\;\;\;\; fdt  \longmapsto  \muu\phi(f)dt.
\end{align*} 

Given a finite-dimensional representation $\rho\colon G\to \GL_V$, we have a morphism $\fg\to \fgl_V$ of Lie algebras, and hence a morphism $\fg_R \to \fgl_V\stens R\cong \End_R(V_R)$ of Lie algebras over $R$ (which is injective if $\rho$ is a closed embedding). For any $X\in \fg_R$, we denote by $X$ the action of $\Lie(\rho)(X)$ on $V_R$ (see Remark~\ref{R:Lie-endo}).
We define the \emph{connection} $\nabla_X$ of $V_R$ (associated to $X$) by
\begin{align*}
\nabla_X=\nabla_{X,V}   \colon  V_R &\longto  V_R \stens_R \Omega_R^1,\\
 v\otimes f &\longmapsto (v\otimes 1)\otimes d(f) + X(v\otimes f)\otimes dt.
\end{align*}
Since $fdt\mapsto f$ gives an isomorphism $\Omega_R^1 \cong R$, we have an isomorphism $\iota\colon V_R\stens_R \Omega^1_R\to V_R$. Let $\Theta_X=\Theta_{X,V}$ be the \emph{differential operator associated to $\nabla_X$} given by the following composition
\[\begin{tikzcd}
V_R \ar[rr,"\nabla_X"]  && V_R\stens_R \Omega^1_R \ar[r,"\iota"]  &V_R.
\end{tikzcd}\]
We have that $\Theta_X(v\otimes f)=v\otimes \partial(f)+X(v\otimes f)$ for all $v\otimes f\in V_R$. Moreover, we have the following easy lemma.

\begin{lem}\label{L:tensor-Theta}
Let $V$ and $W$ be finite-dimensional $G$-representations and let $\alpha\in \Hom_G(V,W)$.  We then have 
\[\alpha_R\circ\Theta_{X,V}=\Theta_{X,W}\circ\alpha_R,\;\;\;\;\text{and}\;\;\;\;\Theta_{X,V\otimes W}=\Theta_{X,V}\otimes \Id_{W_\cR}+ \Id_{V_\cR}\otimes \Theta_{X,W}.\]
\end{lem}
\begin{proof}
The first equality holds since $\alpha_R$ commutes with $X$ (see Remark~\ref{R:Lie-endo}), and the second one follows from a direct computation.
\end{proof}

\begin{cons}\label{C:dlog}
We consider the $R$-algebra morphism
\[\hat\partial\colon  R\longto R[\epsilon],~~~ r\longmapsto r+\partial(r)\epsilon,\]
which induces a morphism $G(\hat\partial)\colon G(R)\to G(R[\varepsilon])$. Notice that $\pi_R\circ \hat\partial =\Id_R$, we then have $G(\pi_R)\circ G(\hat\partial)= \Id_{G(R)}$, in particular, $G(\pi_R)\big(G(\hat\partial) (g)\big)=g$. Identifying $g$ with its image in $G(R[\varepsilon])$ induced by the inclusion $R\to R[\varepsilon],r\mapsto r$, we then have
\[G(\hat\partial)(g)g^{-1}\in \Ker G(\pi_R)=\fg_R.\]
For $g\in G(R)$, we define $\partial  (g)\coloneqq G(\hat\partial) (g)\in G(R[\epsilon])$, and put 
\[\dlog(g)\coloneqq \partial(g)g^{-1}\in \fg_R.\]

\end{cons}

\begin{exmp}
Let $G=\GL_d$ for some $d\in\NN$, and let $B\in G(R)$. We have that $\dlog(B)=I_d+\varepsilon \partial(B)B^{-1}$, where $I_d$ is the $d\times d$ identity matrix and $\partial$ acts on $B$ entry-wise. Using the isomorphism $\Lie(G)(R)=\{I_d+ \varepsilon B \mid B\in \Mat_{d,d}(R)\}\cong \{B\mid B\in \Mat_{d,d}(R)\}$, we may identify $\dlog(B)$ with $\partial(B)B^{-1}$.
\end{exmp}

\begin{defn}
\begin{itemize}
\item[(i)] We define the \emph{gauge transformation}
\[\Gamma_g\colon \fg_R  \longto \fg_R,\;\;\;\;\; X \longmapsto \Ad(g)(X)-\dlog(g),\]
where $\Ad\colon G\to \GL_\fg$ is the adjoint representation.
\item[(ii)] We define $\bB^{\phi,\nabla}(G,R)$ to be the category whose objects are $(g,X)\in G(R)\times \fg_R$ satisfying $X=\Gamma_g(\mu\phi(X))$, and whose morphisms $(g,X) \to (g',X')$ are elements $x\in G(R)$ such that $g'=xg \phi(x^{-1})$ and $X'=\Gamma_x(X)$.
\end{itemize} 
\end{defn}

\begin{lem}\label{L:phi-nabla}
Let $(g,X)\in \bB^{\phi,\nabla}(G,R)$. Then $(V_R,g\phi,\nabla_X)$ is a $(\phi,\nabla)$-module over $R$ for all $V\in \Rep_F(G)$.
\end{lem}
\begin{proof}
Choose a basis $e_1,\cdots,e_d$ for $V$ over $F$ where $d=\dim_F V$. Let $A=(a_{ij})_{i,j} \in \GL_d(R)$ (resp.\ $N=(n_{ij})_{i,j} \in \Mat_{n,n}(R)$) be the representing matrix of $\rho(g)$ (resp.\ $X$). For any $\bv=\sum\limits_{i=1}^d e_i\otimes f_i \in V_\cR$, we compute
\begin{align*}
g\phi(\Theta_X(\bv)) &= g\phi \Big(\ssum\limits_{i=1}^d e_i \otimes \partial(f_i) + \ssum\limits_{j=1}^d e_j \otimes \ssum\limits_{i=1}^d  n_{ji}f_i  \Big)\\
&=\ssum \limits_{j=1}^d e_j \otimes \ssum \limits_{i=1}^d a_{ji}\phi(\partial(f_i)) + \ssum\limits_{k=1}^d e_k \otimes \ssum \limits_{i=1}^d \ssum \limits_{j=1}^d a_{kj}\phi(n_{ji} f_i),
\end{align*}
and
\begin{align*}
\Theta_X (g\phi(\bv)) &=\Theta_X \Big( \ssum\limits_{j=1}^d e_j \otimes \ssum\limits_{i=1}^d a_{ji} \phi(f_i) \Big) \\
&=\ssum\limits_{j=1}^d e_j \otimes \ssum\limits_{i=1}^d \partial(a_{ji}) \phi(f_i) +\ssum\limits_{j=1}^d e_j \otimes \ssum\limits_{i=1}^d a_{ji}\partial( \phi(f_i)) + \ssum\limits_{k=1}^d e_k  \otimes \ssum \limits_{i=1}^d \ssum \limits_{j=1}^d n_{kj} a_{ji}  \phi(f_i).
\end{align*}
Since $\muu\cdot \sum\limits_{j=1}^d e_j \otimes \sum \limits_{i=1}^d a_{ji}\phi(\partial(f_i))= \sum\limits_{j=1}^d e_j \otimes \sum\limits_{i=1}^d a_{ji}\partial( \phi(f_i))$, we have that $\muu\cdot g\phi \circ \Theta_X= \Theta_X \circ g\phi$ if and only if $\mu A \phi(N)=\partial(A)+NA$, i.e., $N=\muu A\phi(N)A^{-1}-\partial(A)A^{-1}$. The last equality holds because of the assumption $X=\Gamma_g\big(\muu\phi(X)\big)$, which completes the proof.

\end{proof}

As a consequence, we may define a functor
\begin{equation}\label{E:B}
\bB^{\phi,\nabla}(G,R) \longto \GphinabR, \;\;\;\;\; (g,X)\longmapsto  \I(g,X),
\end{equation}
where $\I(g,X)(V)\coloneqq (V_R,g\phi,\nabla_X)$. We next show that this functor is an isomorphism. To do this, we need the following elementary descent result.

\begin{lem}\label{L:faithful}
Fix a field $k$, and let $A$ and $B$ be finitely generated $k$-algebras. Let $\rho\colon X\to Y$ be a closed embedding of affine algebraic $k$-schemes for $X=\Spec A$ and $Y=\Spec B$. Let $\iota\colon S\hookrightarrow \tilde S$ be an embedding in $\Alg_k$. Suppose that we are given an element $\tilde z\in X(\tilde S)$ such that $\rho(\tilde z)\in Y(\iota)$, then there exists a unique element $z\in X(S)$ such that $\tilde z=X(\iota)(z)$. 
\end{lem}

\begin{proof}
W have a diagram
\[\begin{tikzcd}
 &&  A\ar[d, densely dotted,"\exists\alpha"] \ar{rrd}{\tilde z} \\
B\ar{urr}{\rho^*} \ar{rr}{\beta}  &&  S  \ar{rr}{\iota}  &&\tilde S
\end{tikzcd}\]
with the outer triangle commutative in which $\rho^*$ is surjective. We prove the lemma by constructing a unique $k$-algebra morphism $\alpha\colon A\to S$ such that $\tilde z=\iota\circ\alpha$, as follows. For any $a\in A$, the surjectivity of $\rho^*$ gives us some $b\in B$ such that $\rho^*(b)=a$. We first claim that $\alpha(a)\coloneqq \beta(b)$ is well-defined. Suppose that $\rho^*(b_1)=\rho^*(b_2)$ for $b_1,b_2\in B$, then $(\tilde z\circ \rho^*)(b_1)=(\tilde z\circ \rho^*)(b_2)$, which implies that 
\[(\iota\circ\beta)(b_1)=(\iota\circ\beta)(b_2).\]
Since $\iota$ is injective, we have that $\beta(b_1)=\beta(b_2)$, as claimed. We then have a map $\alpha$ satisfying $\tilde z\circ \rho^*=\iota\circ\alpha\circ\rho^*$, yielding that $\tilde z=\iota\circ\alpha$ since $\rho^*$ is surjective. In particular, $\alpha$ is a $k$-algebra morphism since $\iota$ is injective and both $\iota$ and $\tilde z=\iota\circ\alpha$ are $k$-algebra morphisms. Finally, we see that $\alpha$ is unique, again because $\iota$ is injective.
\end{proof}

We remark that $\phinabR$ is a rigid tensor category, with tensor products and duals defined in the usual way.

\begin{propo}\label{P:B}
The functor $\bB^{\phi,\nabla}(G,R) \to  \GphinabR$~\eqref{E:B} is an isomorphism of categories.
\end{propo}

\begin{proof}
The proof is similar to that of~\cite[Lemma 9.1.4]{dat}. We first show that the functor is fully faithful. Let $(g,X), (g',X')\in \bB^{\phi,\nabla}(G,R)$, then any  morphism $\eta \colon \I(g,x) \to \I(g',X')$ is an isomorphism, since any morphism of tensor functors between rigid tensor categories is an isomorphism by~\cite[Proposition 1.13]{tc}. By composing $\eta$ with the forgetful functor, we then have an automorphism of the fiber functor $\omega^G\otimes R$. By Corollary~\ref{C:recover}, this automorphism is given by a unique element $x\in G(R)$, which then gives an isomorphism between $(g,X)$ and $(g',X')$, as desired.

It remains to show that, for any $\I\in \GphinabR$ there exists a unique $(g,X)\in \bB^{\phi,\nabla}(G,R)$ such that $\I=\I(g,X)$. For any $(V,\rho)\in \Rep_F(G)$, we write $\I(V,\rho_V)=(V_R,\Phi_V,\nabla_V)$ for a $\phi$-linear map $\Phi_V$ and a connection $\nabla_V$ on $V_R$. The proof consists of two steps.

\textit{Step 1}: There exists a unique $X\in \fg_R$ such that $\nabla_V=\nabla_X$.
We write $\Theta_V$ for be the composition of
\[\begin{tikzcd}
V_R  \ar[r,"\nabla_V"] &V_R \bigotimes \Omega^1_R \ar[r,"\iota"] & V_R,
\end{tikzcd}\]
where $\iota$ is induced by $fdt\mapsto f$, and put $\theta_V\coloneqq \Theta_V-\Id_V\otimes \partial$. It's clear that $\theta_{\id}=0$ where $\id$ denotes the trivial representation. Lemma~\ref{L:tensor-Theta} then implies that the family
\begin{align*}
\big\{\theta_V\colon V_R \to V_R \mid (V,\rho_V)\in \Rep_F(G) \big\}
\end{align*}
of $R$-linear endomorphisms satisfies conditions (i,ii,iii) in Corollary~\ref{C:Lie-tannakian}. We thus find a unique $X\in \fg_R$ such that $\theta_V=\Lie(\rho_V)(X)$ for all $(V,\rho_V)\in \Rep_F(G)$, which implies that $\nabla_V=\nabla_X$.

\textit{Step 2}: There exists a unique $g\in G(R)$ such that $\Phi_V=g\phi$. We first assume $R\in \{\cE^\dagger,\cR\}$. We put $\tilde\Phi_V \coloneqq \Phi_V\otimes \phi$ where $\phi$ is the Frobenius lift on $\tR$ (in particular, $\tR$ is viewed as an $\cR$-module via the $\phi$-equivariant embedding $\psi$ described in \S~\ref{S:slope}). The family 
\[\big\{\lambda_V\coloneqq \tilde\Phi_V \circ (\Id_V\otimes \phi^{-1})\colon V_{\tR}\to V_{\tR} \mid V\in \Rep_F(G) \big \}\]
of $\tR$-linear endomorphisms satisfies conditions (i,ii,iii) in Theorem~\ref{T:recover}, which provides a unique element $\tilde g\in G(\tR)$ such that $\lambda_V=\rho_V(\tilde g)$ for all $(V,\rho_V) \in \Rep_F(G)$. We next reduce $\tilde g$ to an element in $G(\cR)$. We compute
\[\tilde\Phi_V\circ (\Id_V\otimes \phi^{-1})(v\otimes f) =\tilde\Phi_V(v\otimes \phi^{-1}(f))=\rho_V(\tilde g)(v\otimes f),\]
which implies that $\tilde\Phi_V(v\otimes f)=\rho_V(\tilde g)(v\otimes \phi(f))$, and hence, $\tilde\Phi_V=\tilde g \phi$. We now fix a $d$-dimensional faithful representation $(V,\rho_V)$, and an $F$-basis $e_1,\cdots,e_d$ for $V$. Suppose that $\Phi_V(e_i)=\sum\limits_{j=1}^d a_{ji}e_j$, where $a_{ij}\in R$ for all $1\leq i,j \leq d$. Put $A=(a_{ij})_{i,j} \in \GL_d(R)$. Then $\psi(A)=(\psi(a_{ij}))_{i,j}\in \GL_d(\tR)$ describes the $\phi$-linear action of $\tilde\Phi_V$ as well as the $\tR$-linear action $\rho(\tilde g)$ in the basis $e_1\otimes 1,\cdots, e_d \otimes 1$. By replacing $X$ with $G$, $Y$ with $\GL_d$, $S$ with $R$, $\tilde S$ with $\tR$, and $\iota$ with $\psi$ in Lemma~\ref{L:faithful}, we find a unique element $g\in G(R)$ such that $\psi(g)=\tilde g$. It follows that $\Phi_V=g\phi$, as desired. When $R\in \{\tE^\dagger,\tR\}$, we apply the above argument to $\tilde\Phi_V \coloneqq \Phi_V$ dispensing with the reduction method, and we are done.

\end{proof}

\subsection{The pushforward functor}
Let $R\in \{\cE^\dagger,\cR,\tE^\dagger,\tR\}$. For any $g\in G(R)$ and $n\in \NN$, we define
\[[n]_*(g)\coloneqq g\phi(g)\cdots\phi^{n-1}(g)\in G(R),\]
the \emph{$n$-pushforward} of $g$. Notice that $[n]_*(g)\phi^n=(g\phi)^n\in G(R)\rtimes \la\phi\ra$ for all $n\in \NN$. 

We define the \emph{$n$-pushforward functor} by
\[[n]_*\colon \bB^{\phi,\nabla}(G,R) \longto \mathbf B^{\phi^n,\nabla}(G,R), \;\;\;\;\; (g,X) \longmapsto \big([n]_*(g), X \big),\]
and $[n]_*(x)=x$ for all morphisms $x\in \bB^{\phi,\nabla}(G,R)$. The following lemma shows that this functor is well-defined (in particular, faithful).

\begin{lem}\label{L:push}
Let $(g,X)\in \bB^{\phi,\nabla}(G,R)$. We then have $\big([n]_*(g), X \big)\in \mathbf B^{\phi^n,\nabla}(G,R)$ for all $n\in \NN$.

\end{lem}
\begin{proof}
We show by induction on $n$ that
\[X+\dlog\big([n]_*(g)\big)=\muu(\phi^n,t)\Ad\big([n]_*(g)\big)\big(\phi^{n}(X)\big).\]
There is nothing to show when $n=1$. We now assume by the induction hypothesis that
\[X+\dlog\big([n-1]_*(g)\big)=\muu(\phi^{n-1},t)\Ad\big([n-1]_*(g)\big)\big( \phi^{n-1}(X)\big),\]
We notice that $\muu(\phi^{n-1},t)=\muu\phi(\muu)\cdots\phi^{n-2}(\muu)$, and hence 
\[\partial(\phi^{n-1}(f))=\muu\phi(\muu)\cdots\phi^{n-2}(\muu)\phi^{n-1}(\partial(f))=\muu(\phi^{n-1},t)\phi^{n-1}(\partial(f)),~~~~~\forall f\in R,\]
which implies that 
\[\dlog(\phi^{n-1}(g))=\muu(\phi^{n-1},t) \phi^{n-1}(\dlog(g)).\]
On the other hand, since $X+\dlog(g)=\muu \Ad(g)(\phi(X))$, we have
\[\phi^{n-1}(X)+\phi^{n-1}(\dlog(g))= \phi^{n-1}(\muu)\Ad\big(\phi^{n-1}(g)\big)\big(\phi^n(X)\big).\]
We now compute
\begin{align*}
 X+\dlog\big([n]_*(g)\big)
&=X+\dlog\big([n-1]_*(g)\big)+\Ad\big([n-1]_*(g)\big)\big(\dlog(\phi^{n-1}(g))\big)\\
&=\muu(\phi^{n-1},t)\Ad\big([n-1]_*(g)\big) \big(\phi^{n-1}(X)\big)\\
&\;\;\;\; +\muu(\phi^{n-1},t)\Ad\big([n-1]_*(g)\big) \big(\phi^{n-1}(\dlog(g)) \big)\\
&=\muu(\phi^{n-1},t) \Ad\big([n-1]_*(g)\big) \big(\phi^{n-1}(X)+\phi^{n-1}(\dlog(g)) \big)\\
&=  \muu(\phi^{n-1},t)\Ad\big([n-1]_*(g)\big) \big(\phi^{n-1}(\muu)\Ad\big(\phi^{n-1}(g)\big)\big(\phi^n(X)\big)\big)\\
&=\muu(\phi^n,t) \Ad\big([n]_*(g)\big)\big(\phi^{n}(X)\big),
\end{align*}
which proves the lemma.

\end{proof}

In connection with the pushforward functor on $\phi$-modules as recalled in \S~\ref{S:slope}, we state the following lemma resulting from~\cite[Lemma 1.6.3 and Remark 1.7.2]{ked-relative}, which will not be explicitly used in the sequel.

\begin{lem}\label{L:push-slope}
Let $g\in G(R)$. Then $(V_R,g\phi)$ is pure of slope $\mu$ if and only if $(V_R,[n]_*(g)\phi^n)$ is pure of slope $n\mu$ for all $n\in \NN$. Moreover, if $(V_R,g\phi)$ has jumps $\mu_1,\cdots,\mu_l$, then $(V_R,[n]_*(g)\phi^n)$ has jumps $n\mu_1,\cdots,n\mu_l$.
\end{lem}

\subsection{$G$-$(\phi,\nabla)$-modules attached to splittings}
Let $g\in G(\cR)$. We fix a splitting $\xi_g$ of $\HN_g$ by Theorem~\ref{T:Q-splitting}.
\begin{cons}\label{C:Phi'}
Let $(V_\cR,g\phi,\nabla_X)$ be a $(\phi,\nabla)$-module over $\cR$ with the slope filtration
\[0\subseteq V_\cR^{\mu_1} \subseteq \cdots \subseteq V_\cR^{\mu_l}=V_\cR,\]
with jumps $\mu_1<\cdots<\mu_l$. Then $\xi_g(V)$ is the decomposition
\[ V_\cR=\soplus\limits_{i=1}^l V_{\cR,\mu_i}\]
of $\cR$-modules such that $\bigoplus\limits_{i=1}^j V_{\cR,\mu_i} =V_\cR^{\mu_j}$ for $j=1,\cdots,l$. 

\begin{itemize}
\item[(i)] For any $1\leq j\leq l$ and $\bv\in V_{\cR,\mu_j}$, we have $\Phi_g(\bv)\in V_\cR^{\mu_j}$, whence a unique expression $\Phi_g(\bv)=\bsum{i=1}^j \bv_i$ with $\bv_i\in V_{\cR,\mu_i}$. We thus have a $\phi$-linear map
\[\Phi_{g,\mu_j} \colon V_{\cR,\mu_j} \longto V_{\cR,\mu_j},~~~~ \bv \longmapsto \bv_j.\]
We then define $\Phi_g'\coloneqq \bigoplus\limits_{i=1}^l \Phi_{g,{\mu_i}}$. We define
\[\I'(g)(V)\coloneqq (V_\cR,\Phi_g').\]
For a morphism $\alpha\colon V\to W$ of finite-dimensional $G$-modules, we define $\I'(g)(\alpha)\coloneqq \alpha_\cR$.
\item[(ii)] Similarly, for any $1\leq j\leq l$ and $\bv\in V_{\cR,\mu_j}$, we have $\Theta_X(\bv)\in V_\cR^{\mu_j}$, whence a unique expression $\Theta_X(\bv)=\sum_{i=1}^j \bv_i$ with $\bv_i\in V_{\cR,\mu_i}$. We thus have a $K$-linear differential operator
\[\Theta_{X,\mu_j} \colon V_{\cR,\mu_j} \longto V_{\cR,\mu_j},~~~~ \bv \longmapsto \bv_j.\]
We then define $\Theta_X'\coloneqq \bigoplus\limits_{i=1}^l \Theta_{X,{\mu_i}}$.
\end{itemize}
\end{cons}

Notice that $\big( V_{\cR,\mu_1},\Phi_{g,{\mu_1}}\big)=\big( V_\cR^{\mu_1},\Phi_g|_{V_\cR^{\mu_1}} \big)$, and $\big( V_{\cR,\mu_i},\Phi_{g,{\mu_i}}\big)$ is isomorphic to $V_\cR^{\mu_i}/V_\cR^{\mu_{i-1}}$ as $\phi$-modules for $2\leq i\leq l$. Similarly, we have $\big( V_{\cR,\mu_1},\Theta_{X,{\mu_1}} \big)=\big( V_\cR^{\mu_1},\Theta_X|_{V_\cR^{\mu_1}} \big)$, and $\big( V_{\cR,\mu_i},\Theta_{X,\mu_i}\big)$ is isomorphic to $V_\cR^{\mu_i}/V_\cR^{\mu_{i-1}}$ as a differential module for $2\leq i\leq l$.

The remainder of this subsection is devoted to the consequences of Construction~\ref{C:Phi'} (i). We will continue to discuss (ii) in \S~\ref{S:G-phi-nabla}; we will show, in particular, that both constructions assemble to give a $G$-$(\phi,\nabla)$-module over $\cR$. 

\begin{lem}\label{L:I'}
$\I'(g) \colon \Rep_F(G) \to \phi\text{-}\Mod_\cR$ is a $G$-isocrystal.
\end{lem}
\begin{proof}
By Definition~\ref{D:GIsoc}, it amounts to show that $\I'(g)$ is an exact faithful $F$-linear tensor functor. In this proof, we fix $V,W\in\Rep_F(G)$, and suppose the slope filtration of $(V_\cR,g\phi)$ (resp.\ of $(W_\cR,g\phi)$) has jumps $\mu_1<\cdots<\mu_{l_V}$ (resp.\ $\nu_1<\cdots<\nu_{l_W}$).  

We first check the functoriality of $\I'(g)$ (the exactness, faithfulness and $F$-linearity will follow immediately). Given $\alpha\in\Hom_G(V,W)$, we need to show that
\[\alpha_\cR\circ \Phi_g'=\Phi_g'\circ \alpha_\cR.\]
For any fixed $1\leq l \leq {l_V}$, we have that $\alpha_\cR(V_{\cR,\mu_l})\subseteq W_{\cR,\mu_l}$ by Theorem~\ref{T:Q-splitting}. Notice that $W_{\cR,\mu_l}=W_{\cR,\nu_s}$ if $\mu_l=\nu_s$ for some $1\leq s\leq {l_W}$, and $W_{\cR,\mu_l}=0$ otherwise. In the latter case, it is clear that $\alpha_\cR\circ \Phi_g'=\Phi_g'\circ \alpha_\cR=0$ on $V_{\cR,\mu_l}$, and we are done. Suppose now we are in the former case. Let $\bv$ be a non-zero element in $V_{\cR,\mu_l}$. We then have $\Phi_g(\bv)\in V_\cR^{\mu_l}$ and $\alpha_\cR(\bv)\in W_{\cR,\nu_s}$. 
We have unique expressions 
\[ \Phi_g(\bv)= \ssum\limits_{i=1}^l \bv_i,~~~~~\bv_i\in V_{\cR,\mu_i},\]
and 
\[ \alpha_\cR\circ\Phi_g(\bv)=\ssum\limits_{i=1}^s \bw_i,~~~~~\bw_i\in W_{\cR,\nu_i},\]
therefore $\alpha_\cR(\bv_l)=\bw_s$.
We also write
\[ \Phi_g\circ \alpha_\cR(\bv)=\ssum\limits_{i=1}^s \bw_i', \;\;\;\bw_i'\in W_{\cR,\nu_i},\]
we then have $\bw_i=\bw_i'$ for $i=1,\cdots,s$, as $\alpha_\cR\circ\Phi_g=\Phi_g\circ \alpha_\cR$. We thus have $\alpha_\cR\circ\Phi_{g,{\mu_l}}(\bv)=\alpha_\cR(\bv_l)=\bw_s$ and $\Phi_{g,{\nu_s}}\circ\alpha_\cR(\bv)= \bw_s'=\bw_s$, which implies that $\alpha_\cR\circ\Phi_{g,{\mu_l}}=\Phi_{g,{\nu_s}}\circ\alpha_\cR$, as desired.

It remains to show that $\I'(g)$ preserves tensor products. Since $\tau_g$ is a tensor functor, the $(\mu_l+\nu_s)$-th graded piece of $\tau_g(V\otimes W)$ is then
\[ \big(V\stens\limits_F W \big)_{\cR,\mu_l+\nu_s}= \soplus\limits_{\mu_i+\nu_j=\mu_l+\nu_s \atop\\ 1\leq i\leq {l_V}, 1\leq j\leq {l_W}} \big( V_{\cR,\mu_i}\stens\limits_\cR W_{\cR,\nu_j } \big),\]
for all $1\leq l\leq l_V$ and $1\leq s \leq l_W$. It then follows from Construction~\ref{C:Phi'} (i) that 
\[ \Phi'_{g,\mu_l+\nu_s}=\soplus\limits_{\mu_i+\nu_j=\mu_l+\nu_s \atop\\ 1\leq i\leq {l_V},1\leq j\leq {l_W}} \big(\Phi'_{g,\mu_i} \otimes \Phi'_{g,\nu_j}\big),\]
which implies that $\I'(g)(V\stens W)$ coincides with $\I'(g)(V)\stens \I'(g)(W)$ on 
all $(V\stens W)_{\cR,\mu_l+\nu_s}$, whence on $(V\stens W)_{\cR}$. This completes the proof.

\end{proof}

With Lemma~\ref{L:I'}, we imitate \emph{Step 2} in the proof of Proposition~\ref{P:B} and have the following proposition.

\begin{propo}\label{P:z}
There exists a unique element $z\in G(\cR)$ such that $\I'(g)=\I(z)$.
\end{propo}

\subsection{$G$-$(\phi,\nabla)$-modules attached to splittings}\label{S:G-phi-nabla}

We fix $(g,X)\in\mathbf B^{\phi,\nabla}$. We also fix a splitting $\xi_g$ of $\HN_g$ given by Theorem~\ref{T:Q-splitting}.

We now look back at Construction~\ref{C:Phi'} (ii). We  claim that $\Theta_X'-\Id_V\otimes \partial \colon V_R \to V_R$ is $R$-linear for all $(V,\rho_V)\in\Rep_F(G)$. Let $1\leq j\leq l$ and let $v\otimes f\in V_{\cR,\mu_j}$. Suppose that $\Theta_X(v\otimes f)=\bsum_{i=1}^j \bv_i$ with $\bv_i\in V_{\cR,\mu_i}$. Then $\Theta_X'(v\otimes f)=\bv_j$ by construction. Let $f'\in R$. We compute
\begin{align*}
\Theta_X(v\otimes ff') &=v\otimes \partial(f)f'+v\otimes f\partial(f')+X (v\otimes ff')\\
&= \big( v\otimes \partial(f)+ X(v\otimes f) \big)f'+v\otimes f\partial(f')\\
&= \Theta_X(v\otimes f)f' +v\otimes f\partial(f')\\
&=f' \ssum_{i=1}^j \bv_i+v\otimes f\partial(f'),
\end{align*}
which implies that $\Theta_X(v\otimes ff')=f'\bv_j+v\otimes f\partial(f')$. We thus have
\begin{align*}
(\Theta_X'-\Id_V\otimes \partial)(v\otimes ff') &=f'\bv_j+v\otimes f\partial(f')- v\otimes (ff')\\
&= f'\bv_j+v\otimes f\partial(f')-v\otimes \partial(f)f'-v\otimes f\partial(f')\\
&=f' (\bv_j-v\otimes \partial(f))\\
&= f' (\Theta_X'-\Id_V\otimes \partial)(v\otimes f),
\end{align*}
as desired.

The following proposition (and it's proof) is analogous to Lemma~\ref{L:I'}.

\begin{propo}\label{P:X_0}
There exists a unique element $X_0\in \fg_\cR$ such that $\Theta_X'=\Theta_{X_0}$.
\end{propo}
\begin{proof}
For any $(V,\rho_V)\in \Rep_F(G)$, we define $\theta_V\coloneqq \Theta_X'-\Id_V\otimes \partial$. We claim that the family
\begin{align*}
\big\{\theta_V\colon V_\cR\to V_\cR \mid (V,\rho_V)\in \Rep_F(G) \big\}
\end{align*}
of $R$-linear endomorphisms satisfies conditions (i,ii,iii) in Corollary~\ref{C:Lie-tannakian}. The lemma will follow immediately.

It is clear that $\theta_V=0$ if $V=F$ is the trivial $G$-representation. For the remainder of the proof, we fix $(V,\rho_V),(W,\rho_W)\in \Rep_F(G)$, and suppose the slope filtration of $(V_\cR,g\phi)$ (resp.\ of $(W_\cR,g\phi)$) has jumps $\mu_1<\cdots<\mu_{l_V}$ (resp.\ $\nu_1<\cdots<\nu_{l_W}$). Let $\alpha\in\Hom_G(V,W)$. To show that $\theta_V\circ \alpha_\cR =\alpha_\cR\circ \theta_W$, it suffices to show that $\Theta_X'\circ \alpha_\cR =\alpha_\cR \circ \Theta_X'$. Notice that $\alpha_\cR$ respects gradings. Replacing $\Phi_g$ with $\Theta_X$ (possibly with proper decorations) in the second paragraph of the proof of Lemma~\ref{L:I'}, we have the desired result.

It remains to show that
\[\theta_{V\otimes W} =\theta_V\otimes \Id_{W_\cR} +\Id_{V_\cR} \otimes \theta_W.\]
Since $\tau_g$ is a tensor functor, the $(\mu_l+\nu_s)$-th graded piece of $\tau_g(V\bigotimes W)$ is then
\[ \big(V \stens W \big)_{\cR,\mu_l+\nu_s}= \soplus\limits_{\mu_i+\nu_j=\mu_l+\nu_s \atop\\ 1\leq i\leq {l_V}, 1\leq j\leq {l_W}} \big( V_{\cR,\mu_i}\stens_\cR W_{\cR,\nu_j } \big),\]
for all $1\leq l\leq l_V$ and $1\leq s \leq l_W$. It follows from Lemma~\ref{L:tensor-Theta} and Construction~\ref{C:Phi'} that
\[ \Theta'_{X,\mu_l+\nu_s}=\soplus\limits_{\mu_i+\nu_j=\mu_l+\nu_s \atop\\ 1\leq i\leq {l_V}, 1\leq j\leq {l_W}} \big(\Theta'_{X,\mu_i}\otimes \Id_{W_{\cR,\nu_j }}+ \Id_{V_{\cR,\mu_i }} \otimes \Theta'_{X,\nu_j} \big).\]
Let $v\otimes f\otimes w\otimes f'\in V_{\cR,\mu_i}\stens_\cR W_{\cR,\nu_j }$. We compute
\begin{align*}
& \big( \theta_V\otimes \Id_{W_\cR} +\Id_{V_\cR} \otimes \theta_W \big)(v\otimes f\otimes w\otimes f')\\
=& \big(\Theta'_{X,\mu_i}(v\otimes f)-v\otimes \partial(f)\big) \otimes w\otimes f'+ v\otimes f \otimes \big( \Theta'_{X,\nu_j}(w\otimes f')-w\otimes\partial(f') \big)\\
=& \big( \Theta'_{X,\mu_i}\otimes \Id_{W_{\cR,\nu_j }}+ \Id_{V_{\cR,\mu_i }} \otimes \Theta'_{X,\nu_j}\big)(v\otimes f\otimes w\otimes f')-  v\otimes 1\otimes w\otimes \partial(ff') \\
=& \big(\Theta'_{X,\mu_l+\nu_s}-\Id_{V\otimes W}\otimes \partial \big)(v\otimes w\otimes ff')\\
=& \theta_{V\otimes W} (v\otimes w\otimes ff'),
\end{align*}
which completes the proof.

\end{proof}

We now summarize what we have shown thus far. The splitting $\xi_g$ of $\HN_g$ gives a unique element $z\in G(\cR)$ such that $\I'(g)=\I(z)$ by Proposition~\ref{P:z}, and a unique element $X_0\in \fg_\cR$ such that $\Theta'_X=\Theta_{X_0}$ by Proposition~\ref{P:X_0}. These two elements are related as in Proposition~\ref{P:z-X_0} below. We begin with some relative facts.

\begin{rmk}\label{R:para}
This remark is essentially from~\cite[\S 2.1]{cgp}. Let $k$ be a commutative ring with $1$, and let $\fG$ be a reductive $k$-group.  Hereupon, we denote by $\kk(s)$ the residue field of $s$ and $\bar\kk(s)$ the algebraic closure of $\kk(s)$, for all $s\in\Spec k$.
A subgroup $\fP$ of $\fG$ is a \emph{parabolic} (resp.\ \emph{Borel}) subgroup if $\fP$ is smooth and $\fP_{\bar\kk(s)}$ is a parabolic (resp.\ Borel) subgroup of $\fG_{\bar\kk(s)}$, for all $s\in\Spec k$.

Suppose we have a cocharacter $\lambda \colon \GG_m\to \fG$ over $k$. For any $k$-algebra $R$, we let $\GG_{m,R}$ act on $\fG_R$ via the  conjugation
\[
\GG_{m,R}(S) \times \fG_R(S) \longto \fG_R(S),\;\;\;\;(t,x)\longmapsto t.x\coloneqq  \lambda(t)x\lambda(t)^{-1}
\]
for all $R$-algebra $S$. For any $x\in \fG(R)$, we have an \emph{orbit map} $\alpha_x\colon \GG_{m,R} \to \fG_R$ given by
\[\alpha_x\colon \GG_{m,R}(S) \longto \fG_R(S),\;\;\;\; t \longmapsto  t.x\]
for all $R$-algebras $S$. Let $\dA^1$ be the affine $k$-line. We say that the \emph{limit} $\lim\limits_{t \to 0} t.x$ exists if $\alpha_x$ extends (necessarily uniquely) to a morphism $\tilde\alpha_x\colon \dA^1_R \to \fG_R$ of affine $R$-schemes, and put $\lim\limits_{t \to 0} t.x \coloneqq \tilde\alpha_x(0)\in \fG_R(R)$. We define
\[P_\fG(\lambda) (R)\coloneqq \big\{x\in \fG(R) \mid \lim\limits_{t \to 0} t.x~\text{exists} \big\},\]
\[U_\fG(\lambda) (R)\coloneqq \big\{x\in \fG(R) \mid \lim\limits_{t \to 0} t.x=1 \big\},\]
and
\[Z_\fG(\lambda) (R)\coloneqq P_\fG(\lambda)(R)\cap P_\fG(-\lambda)(R),\]
where $-\lambda$ is the inverse of $\lambda$. Then $P_\fG(\lambda)$ is a closed $k$-subgroup of $\fG$ (\cite[Lemma 2.1.4]{cgp}), $U_\fG(\lambda)$ is an affine algebraic $k$-normal subgroup of $P_\fG(\lambda)$, and $Z_\fG(\lambda)$ is the centralizer of the $\GG_m$-action in $\fG$ (\cite[Lemma 2.1.5]{cgp}). By~\cite[Proposition 2.1.8 (3)]{cgp}, these subgroups are smooth because $\fG$ is smooth.

It follows from the definitions that the formations of $P_\fG(\lambda),U_\fG(\lambda)$ and $Z_\fG(\lambda)$ commute with any base extension on $k$. In particular, for every $s\in\Spec k$ we have $P_\fG(\lambda)_{\bar\kk(s)}
=P_{\fG_{\bar\kk(s)}}(\lambda_{\bar\kk(s)})$, which is a parabolic subgroup of $\fG_{\bar\kk(s)}$ by~\cite[Proposition 8.4.5]{springer}. Hence, $P_\fG(\lambda)$ is a parabolic $k$-group. 

By~\cite[Proposition 2.1.8 (2)]{cgp}, the multiplication map gives an isomorphism
\[U_\fG(\lambda) \rtimes Z_\fG(\lambda) \longto P_\fG(\lambda)\]
of affine algebraic $k$-groups. 

Now let $\GG_m$ act on $\fg=\Lie(\fG)(k)$ through the adjoint representation. We then have $\fg=\soplus\limits_{n\in \ZZ} \fg_n$, where $\fg_n=\{X\in \fg \mid t.X=t^nX,\forall t\in \GG_m\}$ for all $n\in \ZZ$. We have $\Lie\big(Z_\fG(\lambda) \big) =\fg_0$ (which is the centralizer of the $\GG_m$-action on $\fg$), $\Lie\big(U_\fG(\lambda) \big) =\soplus\limits_{n>0} \fg_n$, and $\Lie\big(P_\fG(\lambda) \big) =\soplus\limits_{n\geq 0} \fg_n$. In particular, we have the following decomposition
\begin{equation}\label{E:Lie-decomp}
\Lie\big(P_\fG(\lambda) \big)=\Lie\big(Z_\fG(\lambda) \big)\textstyle\soplus \Lie\big(U_\fG(\lambda) \big).
\end{equation}
\end{rmk}

\begin{lem}\label{L:Lie(U)}
With the notion above, we have
\[Z-\Ad(u)(Z)\in \Lie\big( U_\fG(\lambda)\big),\]
for all $u\in U_\fG(\lambda)(k)$ and $Z\in \Lie\big( Z_\fG(\lambda)\big)$.
\end{lem}
\begin{proof}
Recall that $Z\in Z_\fG(\lambda)(k[\varepsilon])$ by definition; we may also view $u$ as an element in $U_\fG(\lambda)(k[\varepsilon])$ via the inclusion $\iota\colon k\hookrightarrow k[\varepsilon]$. By the definition of the adjoint representation, we have 
\[Z-\Ad(u)(Z)=Z(uZu^{-1})^{-1}=ZuZ^{-1}u^{-1} \in P_\fG(\lambda)(k[\varepsilon]).\]
Because $U_\fG(\lambda)$ is normal in $P_\fG(\lambda)$, we have that $ZuZ^{-1}\in U_\fG(\lambda)(k[\varepsilon])$, and so is $ZuZ^{-1}u^{-1}$. Consider the following commutative diagram
\[\begin{tikzcd}
U_\fG(\lambda)(k[\varepsilon])\ar[d]   \ar[rr,hook] && P_\fG(\lambda)(k[\varepsilon])\ar[d]\\
U_\fG(\lambda)(k)  \ar[rr,hook] && P_\fG(\lambda)(k)
\end{tikzcd}\]
Since both $Z$ and $uZ^{-1}u^{-1}$ lie in the kernel of the right vertical map, so does their product $ZuZ^{-1}u^{-1}$. Hence, $ZuZ^{-1}u^{-1}\in U_\fG(\lambda)(k[\varepsilon])$ lies in the kernel of the left vertical map. The lemma then follows.
\end{proof}

\begin{propo}\label{P:z-X_0}
Let $z\in G(\cR)$ and $X_0\in \fg_\cR$ be the unique elements given by Proposition~\ref{P:z} and Proposition~\ref{P:X_0}, respectively. We have $X_0=\Gamma_z \big(\muu\phi(X_0)\big)$. In particular, $\I(z,X_0)$ is a $G$-$(\phi,\nabla)$-module over $\cR$.
\end{propo}

\begin{proof}
The second assertion follows from the first assertion and Lemma~\ref{L:phi-nabla}. For the first assertion, we need to show 
\begin{equation}\label{E:z-X_0}
X_0=\muu \cdot \Ad(z)\big(\phi(X_0)\big)-\dlog(z).
\end{equation}
It suffices to show~\eqref{E:z-X_0} with both sides understood as elements in $\End_\cR(V_\cR)$ for some faithful representation $(V,\rho)\in\Rep_F(G)$. Suppose that $\dim_F V=d$, and suppose that $\nu_g(V)$ is the decomposition $V_\cR=\bigoplus\limits_{i=1}^l V_{\cR,\mu_i}$. We choose for each graded-piece $V_{\cR,\mu_i}$ a basis. They altogether give a basis $\bv_1,\cdots,\bv_d$ of $V_\cR$, in which $\Phi_g$ acts via a block-upper-triangular matrix 
\begin{equation*}\footnotesize
A=
\begin{pmatrix}
A_1  & &\\   
& A_2  &~~~~\Ast\\
 && \ddots  \\
&&     &A_l
\end{pmatrix}\in\GL_d(\cR),
\end{equation*}
where each $A_i$ is an $m_i$ by $m_i$ invertible matrix with $m_i$ the multiplicity of $\mu_i$. Then $\Phi_z$ acts in this basis via $Z\coloneqq \Diag(A_1,\cdots,A_l)$. Likewise, $\Theta_X$ acts in the  basis $\bv_1,\cdots,\bv_d$ via a block-upper-triangular matrix  
\begin{equation*}\footnotesize
N=
\begin{pmatrix}
N_1  & &\\   
& N_2  &~~~~\Ast \\
 && \ddots  \\
&&     &N_l
\end{pmatrix}\in\Mat_{d,d}(\cR),
\end{equation*}
where each $N_i$ is an $m_i$ by $m_i$ matrix, and $\Theta_{X_0}$ acts via $\overline N\coloneqq \Diag(N_1,\cdots,N_l)$. Write $A=ZU$ for $U\in \GL_d(\cR)$, and $N=\overline N+N_+$ for $N_+\in\Mat_{d,d}(\cR)$. Since $X=\Gamma_g\big(\muu\phi(X)\big)$, we have $N=\muu\cdot A \phi(N) A^{-1}-\partial(A)A^{-1}$, and then
\begin{align*}
 \overline N+N_+
=& \muu\cdot (UZ)(\phi(\overline N+N_+))(UZ)^{-1}-\partial(UZ)(UZ)^{-1}\\
=& \muu\cdot(UZ)\phi(\overline N)Z^{-1}U^{-1}+\muu\cdot(UZ)\phi(N_+)Z^{-1}U^{-1}-\partial(U)U^{-1}-U\partial(Z)Z^{-1}U^{-1}.
\end{align*}
Applying $\Ad(U^{-1})$ on both sides, we then have
\begin{align*}
&\muu\cdot Z\phi(\overline N)Z^{-1}-\partial(Z)Z^{-1}+\muu\cdot Z\phi(N_+)Z^{-1}-U^{-1}\partial(U)\\
=& U^{-1}\overline NU+U^{-1}N_+U
= \overline N -(\overline N-U^{-1}\overline NU-U^{-1}N_+U).
\end{align*}

We claim that $\muu\cdot Z\phi(\overline N)Z^{-1}-\partial(Z)Z^{-1}=\overline N$. Put $\lambda_{\rho,g}\coloneqq \rho\circ\lambda_g \colon\GG_{m,\cR} \to \GL_{V,\cR}$, where $\lambda_g\colon \GG_{m,\cR} \to G_{\cR}$ is the slope morphism defined in Construction~\ref{C:slope-morphism}. Identifying $\GL_{V,\cR}$ with $\GL_{d,\cR}$ via the basis $\bv_1,\cdots,\bv_d$ given in the preceding paragraph, and letting $\fG=\GL_{d,\cR}$, we then have an isomorphism
\[U_\fG(-\lambda_{\rho,g}) \rtimes Z_\fG(-\lambda_{\rho,g}) \cong P_\fG(-\lambda_{\rho,g})\]
of affine algebraic $\cR$-groups. Since $\mu_1< \cdots< \mu_l$, we have
\[A\in P_\fG(-\lambda_{\rho,g})(\cR),~ U\in U_\fG(-\lambda_{\rho,g})(\cR),~ Z\in Z_\fG(-\lambda_{\rho,g})(\cR);\]
\[N\in \Lie\big( P_\fG(-\lambda_{\rho,g})\big),~ N_+\in \Lie\big( U_\fG(-\lambda_{\rho,g})\big),~\overline N\in \Lie\big( Z_\fG(-\lambda_{\rho,g})\big).\]
It follows from Lemma~\ref{L:Lie(U)} that $\overline N-U^{-1}\overline NU \in \Lie \big( U_\fG(-\lambda_{\rho,g})\big)$. In particular, we have $\overline N-U^{-1}\overline NU-U^{-1}N_+U \in \Lie \big( U_\fG(-\lambda_{\rho,g})\big)$. On the other hand, it is clear that $\muu\cdot Z\phi(\overline N)Z^{-1}-\partial(Z)Z^{-1}\in \Lie\big( Z_\fG(-\lambda_{\rho,g})\big)$ and $\muu\cdot Z\phi(N_+)Z^{-1}-U^{-1}\partial(U)\in \Lie \big( U_\fG(-\lambda_{\rho,g})\big)$. By decomposition~\eqref{E:Lie-decomp}, we have $\muu\cdot Z\phi(\overline N)Z^{-1}-\partial(Z)Z^{-1}=\overline N$, and the desired equality~\eqref{E:z-X_0} follows.

\end{proof}

Recall that the least common denominator $d_g$ of $g$ is constructed in ~Construction~\ref{C:d_g}, and $\lambda_g \colon \GG_{m,\cR} \to G_\cR$ is the slope morphism (see Construction~\ref{C:slope-morphism}). We next reduce the $G$-$(\phi,\nabla)$-module $(z,X_0)$ over $\cR$ to a unit-root one by applying the pushforward functor $[d_g]_*$ and \emph{twisting} by $\lambda_g(\pi^{-1})$.

\begin{cor}\label{C:unit-root}
$\I \big(\lambda_g(\pi^{-1})[d_g]_*(z),X_0 \big)$ is a unit-root $G$-$(\phi^{d_g},\nabla)$-module over $\cR$.
\end{cor}

\begin{proof}
For any $V\in \Rep_F(G)$, it suffices to show that $(V_\cR,[d_g]_*(z)\phi^{d_g},\nabla_{X_0})$ is unit-root. By Proposition~\ref{P:z-X_0} and Lemma~\ref{L:push}, $(V_\cR,[d_g]_*(z)\phi^{d_g},\nabla_{X_0})$ is a $(\phi^{d_g},\nabla)$-module over $\cR$. Equivalently, we have $\Theta_{X_0}\circ \Phi_z^{d_g}= \muu\cdot \Phi_z^{d_g}\circ \Theta_{X_0}$. Suppose that $(V_\cR,g\phi)$ has jumps $\mu_1,\cdots,\mu_l$, then $\big( V_\cR,[d_g]_*(z)\phi^{d_g} \big)$ has jumps $d_g\mu_1,\cdots,d_g\mu_l$ by Lemma~\ref{L:push-slope}. For any $1\leq i \leq l$, $\rho(\lambda_g(\pi^{-1}))$ acts via multiplication by $\pi^{-d_g\mu_i}\in K$ on the graded-piece $V_{\cR,\mu_i}$, which implies that $\big( V_{\cR,\mu_i},\lambda_g(\pi^{-1})[d_g]_*(z)\phi^{d_g} \big)$ is unit-root. It follows from~\cite[Proposition 4.6.3 (a)]{ked-revisit} that $\big( V_\cR,\lambda_g(\pi^{-1})[d_g]_*(z)\phi^{d_g} \big)$ is unit-root. Moreover, since $\Theta_{X_0}$ is $K$-linear, we have
\[\Theta_{X_0}\circ \rho(\lambda_g(\pi^{-1}))\circ \Phi_z^{d_g}= \rho(\lambda_g(\pi^{-1}))\circ \Theta_{X_0}\circ \Phi_z^{d_g}= \muu\cdot \rho(\lambda_g(\pi^{-1}))\circ \Phi_z^{d_g}\circ \Theta_{X_0},\]
which completes the proof.

\end{proof}

\subsection{A $G$-version of the $p$-adic local monodromy theorem}

Let $L$ be a finite separable extension of $\kk \lb t \rb$, and let $\cE^\dagger_L$ be the unique unramified extension of $\cE^\dagger$ with residue field $L$. We put $\cR_L \coloneqq \cR\stens_{\cE^\dagger} \cE^\dagger_L$.

We put
\[\cE^{\dagger,\nr}\coloneqq \varinjlim_L \cE^\dagger_L, \;\;\;\text{and} \;\;\; \cB_0\coloneqq \varinjlim_L \cR_L \cong \cR\stens_{\cE^\dagger} \cE^{\dagger,\nr},
\]
where $L$ runs through all finite separable extensions of $\kk \lb t \rb$. In fact, $\cE^{\dagger,\nr}$ is the maximal unramified extension of $\cE^\dagger$ with residue field $\kk \lb t \rb^{\sep}$, the separable closure of $\kk  \lb t \rb$.

The main result of this paper is:

\begin{thm}\label{T:G-mono}
Let $G$ be a connected reductive $F$-group and let $(g,X)\in \mathbf B^{\phi,\nabla}(G,\cR)$. Then there exists a finite separable extension $L$ of $\kk \lb t \rb$ and an element $b\in G(\cR_L)$ such that $\Gamma_b(X)\in \Lie \big(U_{G_\cR}(-\lambda_g)\big)_{\cR_L}$.
\end{thm}

We will make use of the following lemma, which is often mentioned as Steinberg's theorem. The theory of fields of cohomological dimension $\leq 1$ can be found in, e.g.,~\cite[II. \S 3]{ga-co}; for us, the most important example will be a henselian discretely valued field of characteristic $0$ with algebraically closed residue field (see~\cite[II. \S 3.3]{ga-co}).

\begin{lem}[{\cite[Theorem 1.9]{steinberg}}] \label{L:steinberg}
Suppose that $k$ is a field of cohomological dimension $\leq 1$ and $\fG$ is a connected reductive $k$-group, then have $H^1(k,\fG)=1$.
\end{lem}

We also recall that the formations of the subgroups given in Remark~\ref{R:para} commute with base extension.

\begin{proof}[Proof of Theorem~\ref{T:G-mono}]
Let $z\in G(\cR)$ and $X_0\in \fg_\cR$ be the unique elements given by Proposition~\ref{P:z} and Proposition~\ref{P:X_0}, respectively.

Let $(V,\rho)$ be a $d$-dimensional $G$-representation (not necessarily faithful). Suppose the slope filtration of $(V_\cR,g\phi)$ has jumps $\mu_1,\cdots,\mu_l$. Suppose that $\xi_g(V)=\bigoplus\limits_{i=1}^l V_{\cR,\mu_i}$, we put $d_i\coloneqq \rk_\cR (V_{\cR,\mu_i})$ for all $i$. In the proof of Corollary~\ref{C:unit-root} we see that $\big(V_{\cR,\mu_i},\lambda_g(\pi^{-1})[d_g]_*(z)\phi^{d_g},\nabla_{X_0} \big)$ is a unit-root $(\phi,\nabla)$-module over $\cR$ for all $1\leq i\leq l$. Let $\Phi_z=z\phi$ and let $\Theta_{X_0}\colon V_\cR\to V_\cR$ be the differential operator associated to $\nabla_{X_0}$. Then $\Phi_z$ (resp.\ $\Theta_{X_0}$) may be extended to $V\stens_F \cB_0$, which is still denoted by $\Phi_z$ (resp.\ $\Theta_{X_0}$). By the unit-root $p$-adic local monodromy theorem~\cite[Theorem 6.11]{ked-ann}, we find:
\begin{itemize}
\item[(i)] a finite separable extension $L(V)$ of $\kk \lb t \rb$;
\item[(ii)] for each $1\leq i\leq l$ a basis $\bw_1^{(i)},\cdots,\bw_{d_i}^{(i)}$ for $V_{\cR,\mu_i}\bigotimes_\cR \cR_{L(V)}$ over $\cR_{L(V)}$ such that $\Theta_{X_0}(\bw_j^{(i)})=0$ for all $1\leq j\leq d_i$.
\end{itemize}
Then, for each $1\leq i\leq l$, we have that
\[W_i\coloneqq (V_{\cR,\mu_i}\stens\limits_\cR \cB_0)^{\Theta_{X_0}=0}= \big\{ x\in V_{\cR,\mu_i}\stens\limits_\cR \cB_0 \mid \Theta_{X_0}(x)=0 \big\}\]
is a $d_i$-dimensional $K^{\nr}$-vector space spanned by $\bw_1^{(i)},\cdots,\bw_{d_i}^{(i)}$. In particular, we have
\[  (V_{\cB_0} )^{\Theta_{X_0}=0}= \big\{ x\in V_{\cB_0} \mid \Theta_{X_0}(x)=0 \big\}=\soplus\limits_{i=1}^l W_i,\]
which is a $d_i$-dimensional $K^{\nr}$-vector space.

We now have two $K^{\nr}$-valued fiber functors
\[\omega_1=\omega^G \otimes K^{\nr}\colon \Rep_F(G) \longto \Vect_{K^{\nr}},\;\;\;\; V \longmapsto V\otimes K^{\nr},\]
and
\[\omega_2 \colon \Rep_F(G) \longto \Vect_{K^{\nr}},\;\;\;\; V\longmapsto (V_{\cB_0} )^{\Theta_{X_0}=0}.\]
Moreover, we have an action
\[\uIsom^\otimes(\omega_1,\omega_2) \times \uAut^\otimes(\omega_1)\longto \uIsom^\otimes(\omega_1,\omega_2)\]
of $\uAut^\otimes(\omega_1)$ on $\uIsom^\otimes(\omega_1,\omega_2)$, given by pre-composition. We note that $\uAut^\otimes(\omega_1)=\uAut^\otimes(\omega^G\otimes K^{\nr})\cong G_{K^{\nr}}$,\footnote{For this isomorphism, we refer to the discussion above Proposition~\ref{P:Z-splitting}.} so $\uIsom^\otimes(\omega_1,\omega_2)$ may be viewed as a $G_{K^{\nr}}$-torsor over $K^{\nr}$. By Lemma~\ref{L:steinberg}, we have $H^1(K^{\nr},G_{K^{\nr}})=1$. Thus, $\uIsom^\otimes(\omega_1,\omega_2)$ is isomorphic to the trivial $G_{K^{\nr}}$-torsor over $K^{\nr}$, i.e., we have $\uIsom^\otimes(\omega_1,\omega_2)_{K^{\nr}} \cong G_{K^{\nr}}$. 

On the other hand, we have an isomorphism $\gamma\colon \omega_2\otimes \cB_0 \to \omega_1 \otimes \cB_0$ of tensor functors, induced by the $\cB_0$-linear extension of the inclusion
\[\begin{tikzcd}
(V_{\cB_0})^{\Theta_{X_0}=0} \ar[r,hook] &V_{\cB_0}
\end{tikzcd}\]
for all $(V,\rho)\in\Rep_F(G)$. We now fix $\beta\in\uIsom^\otimes(\omega_1,\omega_2)(K^{\nr})$, we then have an element $\tilde \beta\coloneqq \gamma\circ \beta_{\cB_0} \in \uAut^\otimes(\omega_1\otimes \cB_0)(\cB_0)=G_{\cB_0}$. Let $b\in G(\cB_0)$ be the inverse of the image of $\tilde \beta$ under the following isomorphism
\[\uAut^\otimes(\omega_1\otimes \cB_0)(\cB_0) \longto G_{\cB_0}(\cB_0)=G(\cB_0).\]
Since $F[G]$ is finitely presented over $F$, the functor $\Hom_{\Alg_F}(F[G],\func)$ commutes with colimits. We have
\[G(\cB_0)=G(\varinjlim_L \cR_L)=\varinjlim_L G(\cR_L),\]
where $L$ runs over all finite separable extensions of $\kk \lb t \rb$, we thus find a finite separable extension $L$ of $\kk \lb t \rb$ such that $b\in G(\cR_L)$.

For any $(V,\rho)\in \Rep_F(G)$, it follows from the construction of $b$ that the automorphism $\rho(b^{-1})\colon V_{\cB_0} \to V_{\cB_0}$ factors through $(V_{\cB_0})^{\Theta_{X_0}=0} \otimes \cB_0$. Notice that $\Theta_{X_0}$ and $X_0$ agree on $\omega_1(V)=V_{K^{\nr}}$. Therefore, we have
\begin{align}\label{E:X_0}
\rho(b)X_0\rho(b^{-1})-\partial(\rho(b))\rho(b^{-1})=0.
\end{align}

We now fix a faithful representation $(V,\rho)$. The equality~\eqref{E:X_0} then implies
\[\Gamma_{b}(X_0)=0.\]
Put $X_1\coloneqq X-X_0\in \fg_\cR$, we then have
\begin{align*}
\Gamma_{b}(X) &=\Ad(b)(X_0+X_1)-\dlog(b)\\
&= \Ad(b)(X_0)-\dlog(b)+\Ad(b)(X_1)\\
&=\Gamma_{b}(X_0)+\Ad(b)(X_1)\\
&=\Ad(b)(X_1).
\end{align*}
Conserving the notation as in the second paragraph, $\Theta_X=\rho(b) X_1 \rho(b^{-1})$ acts in the basis $\bw_1^{(1)},\cdots,\bw_{d_1}^{(1)},\cdots,\bw_1^{(l)},\cdots,\bw_{d_l}^{(l)}$ via a matrix of the form
\begin{equation*}\footnotesize
\begin{pmatrix}
 0  & &\\   
&& 0  &~~~~\Ast \\
 &&& \ddots  \\
&&& &    0
\end{pmatrix}\in\Mat_{d,d}(\cR_L).
\end{equation*}
Here, the $i$-th $0$ in the diagonal denotes the zero matrix of size $d_i\times d_i$. Hence, $\Gamma_{b}(X)\in \Lie \big(U_{G_{\cR_L}}(-\lambda_{g,\cR_L})\big)=\Lie\big( U_{G_\cR}(-\lambda_g)_{\cR_L} \big)=\Lie \big(U_{G_\cR}(-\lambda_g)\big)_{\cR_L}$, as desired.

\end{proof}

\bibliography{G-phi-nabla}
\bibliographystyle{plain}

\par
\medskip
\begin{tabular}{@{}l@{}}
	\textsc{Department of Mathematics, East China Normal University,} \\
	\textsc{500 Dongchuan Road, Shanghai, 200241 P.R.\ China} \\[1.5pt]
	\textit{E-mail address}: \texttt{syye@math.ecnu.edu.cn}
\end{tabular}

\end{document}